%% file: main.tex
\documentclass[10pt]{article}

\usepackage{geometry}

\usepackage{amsmath,amssymb}
\usepackage{nccmath,mathtools,mathrsfs}
\usepackage{graphicx}

\usepackage{amsthm}

\DeclarePairedDelimiterX\setv[2]{\{}{\}}{#1 \;\delimsize\vert\; #2}

\usepackage{enumitem}
\setlist{leftmargin=12pt, labelsep=4pt, labelwidth=8pt}

\usepackage{algorithm,algpseudocode}

\usepackage{xcolor}

\usepackage{comment}

\newtheorem{theorem}{Theorem}

\newtheorem{lemma}{Lemma}
\newtheorem{corollary}{Corollary}

\newtheorem{assumption}{Assumption}

\theoremstyle{definition}

\theoremstyle{remark}
\newtheorem{remark}{Remark}

\DeclareMathOperator*{\argmin}{arg\,min}
\DeclareMathOperator*{\argmax}{arg\,max}                 

\title{\bf
Zeroth-Order Feedback Optimization in Multi-Agent Systems: Tackling Coupled Constraints}

\date{}

\author{Yingpeng Duan and Yujie Tang% <-this % stops a space
\thanks{The authors are with the College of Engineering, Peking University, China (e-mail: {\tt\small zxycy@stu.pku.edu.cn, yujietang@pku.edu.cn}).}%
}

\begin{document}

\maketitle

%%%%%%%%%%%%%%%%%%%%%%%%%%%%%%%%%%%%%%%%%%%%%%%%%%%%%%%%%%%%%%%%%%%%%%%%%%%%%%%%
\begin{abstract}
This paper investigates distributed zeroth-order feedback optimization in multi-agent systems with coupled constraints, where each agent operates its local action vector and observes only zeroth-order information to minimize a global cost function subject to constraints in which the local actions are coupled. %This problem has numerous applications in areas such as networked control and resource allocation. Additionally, in many scenarios, gradient information is not explicitly available, requiring agents to use outcomes (feedback) from their local costs for optimization. 
Specifically, we employ two-point zeroth-order gradient estimation with delayed information to construct stochastic gradients, and leverage the constraint extrapolation technique and the averaging consensus framework to effectively handle the coupled constraints. We also provide convergence rate and oracle complexity results for our algorithm, characterizing its computational efficiency and scalability by rigorous theoretical analysis. Numerical experiments are conducted to validate the algorithm's effectiveness.
\end{abstract}

%%%%%%%%%%%%%%%%%%%%%%%%%%%%%%%%%%%%%%%%%%%%%%%%%%%%%%%%%%%%%%%%%%%%%%%%%%%%%%%%
\section{Introduction}

Distribution optimization has gained significant attention due to its wide-ranging applications in networked systems, large-scale machine learning, and beyond \cite{yang2019survey,nedic2020distributed,mao2019finite,wen2022distributed}. In distributed optimization problems, the systems consist of agents that aim to achieve a global objective through communicating and coordinating with other agents via a network. One of the challenges in the study of distributed optimization is handling coupled constraints among the agents. Coupled constraints occur when the feasible decisions of one agent are dependent on the actions of the other agents, making optimization problem more intricate.

A variety of algorithms have been developed to address optimization problems with constraints, among which a common and effective class of approaches is the primal-dual methods. The primal-dual methods have a long-standing history that can date back to the seminal work~\cite{arrow1958studies}, while still being an active topic in optimization in recent years due to their flexibility and scalability in large-scale problems~\cite{nemirovski2004prox}. Some recent works on primal-dual methods include, 
e.g., \cite{xu2021first} which proposed a linearized augmented Lagrangian method for solving composite convex problems with both
affine equality and smooth nonlinear inequality constraints. The work \cite{hamedani2021primal} proposed an accelerated primal-dual algorithm with momentum to solve saddle-point problems with more general coupling between the primal and the dual variables. The work~\cite{boob2023stochastic} proposed a constraint extrapolation technique for handling both convex and nonconvex functional constraints with stochastic gradient oracles.
Besides centralized optimization, researchers have also investigated how to handle functional or coupled constraints in distributed optimization. For instance, \cite{chang2014distributed} presented a distributed primal-dual algorithm leveraging a consensus framework. \cite{falsone2017dual} developed a distributed algorithm based on dual decomposition and proximal minimization for minimizing a separable objective function with constraint $\sum_i g_i(x_i)\leq 0$. Other works that employed primal-dual techniques to efficiently manage coupled constraints in various distributed optimization scenarios include \cite{notarnicola2019constraint,liang2019distributed,huang2021primal, wu2022distributed}, etc.

Despite these advancements, a significant limitation persists in many practical applications where agents lack access to the explicit form of their local cost functions or their gradients. Such cases, often studied under the realm of gradient-free or zeroth-order optimization, require agents to rely solely on observed outcomes (feedback) rather than analytical gradients to make decisions. Zeroth-order techniques have become increasingly prominent in recent research. In the centralized optimization context, \cite{nemirovskij1983problem,flaxman2004online} studied one-point zeroth-order gradient estimation techniques for gradient-free optimization. The works \cite{duchi2015optimal,nesterov2017random,shamir2017optimal} further advanced the field with two-point zeroth-order gradient estimation. Some more recent advancements include \cite{zhang2022new} which proposed one-point gradient estimation using residual feedback, \cite{ren2023escaping} which studied escaping saddle points with zeroth-order methods in nonconvex optimization, and \cite{nguyen2023stochastic} which extended the constraint extrapolation technique to zeroth-order optimization. There is also rich existing literature on distributed zeroth-order optimization, including \cite{hajinezhad2019,tang2020distributed,yuan2024distributed} that considered consensus optimization problems, \cite{Shen2021zeroth} that investigated asynchronous zeroth-order algorithms in a distributed setting, and \cite{tang2023zeroth,zheng2023zeroth} that studied multi-agent zeroth-order feedback optimization.

The problem setup considered in this paper is closely related to \cite{tang2023zeroth}, where each agent $i$ can only control its own local action vector $x^i$, and each local cost function $f_i$ is influenced by the joint action profile $x = (x^1, ..., x^n)$. However, \cite{tang2023zeroth} did not address the challenge of coupled constraints, which are prevalent in many practical multi-agent systems. This paper considers coupled constraints of the form $\sum_{i=1}^n g_{ij}(x^i)\le 0$, and aims to develop a zeroth-order feedback optimization algorithm that not only addresses the zeroth-order feedback requirements in cooperative multi-agent systems but also handle coupled constraints effectively.

\subsection{Our Contributions}
In this paper, we propose a distributed zeroth-order feedback optimization algorithm specifically designed to handle multi-agent systems with coupled constraints. We consider a problem setup that extends the one in \cite{tang2023zeroth} to incorporate coupled constraints of the form $\sum_{i=1}^n g_{ij}(x^i)\leq 0$. Our algorithm employs the zeroth-order gradient estimation technique from \cite{nesterov2017random}, and also utilizes the constraint extrapolation technique in \cite{boob2023stochastic,nguyen2023stochastic}. The resulting algorithm achieves efficient coordination without requiring access to explicit gradient information, while also being able to effectively manage the coupled constraints.

We provide detailed complexity analysis of our algorithm, including its convergence in terms of the objective value gap and constraint violation assessment. Specifically, we show that the number of zeroth-order queries per agent needed for the objective value gap and the constraint violation to fall below $\epsilon>0$ is bounded by $O(d/\epsilon^2)$, where $d$ is the problem dimension. This complexity bound matches typical results for stochastic first-order and zeroth-order methods.

%Due to page limit, we omit the proofs of our theoretical results. Detailed proofs can be found in our online report~\cite{online_report}.

\section{Problem Formulation}
In this paper, we consider the situation where there are $n$ agents connected by a communication network. The topology of the network is specified by the undirected graph
$\mathcal{G} = (\mathcal V, \mathcal E)$, where $\mathcal{V} \coloneqq \{1, ..., n\}$ represents the set of
nodes, and $\mathcal{E}\subseteq\mathcal{V}\times\mathcal{V}$ represents the set of edges. At each time step, agent $i$ can exchange
information only with its neighbors in the communication network. We assume that $\mathcal{G}$ is connected, and use $b_{ij}$ to denote the distance 
(the length of the shortest path) between nodes $i$ and $j$.

Additionally, each agent $i$ is 
associated with a local action vector 
$x^i \in \mathcal X_i$, where $\mathcal{X}_i\subseteq\mathbb{R}^{d_i}$ represents the set of all possible local actions for agent $i$. We shall assume that each $\mathcal{X}_i$ is convex and compact. The joint action profile of the 
group of agents is represented by 
$x \coloneqq (x^1, ..., x^n)\in \mathcal X$, where $\mathcal X \coloneqq \prod_{i=1}^n \mathcal X_i \subseteq\mathbb{R}^d$ and we denote
$d \coloneqq \sum_{i=1}^n d_i$. The goal of the group of agents is to solve the following optimization problem:
\begin{equation}
    \begin{aligned}
        \operatorname*{minimize}_{x^1,\ldots,x^n} \quad &  f_0(x) = \frac{1}{n}\sum_{i=1}^n f_i(x^1, ..., x^n)
        \\
        \textrm{subject to}\quad  &
        \sum_{i=1}^n g_{ij}(x^i) \leq 0,\ \ j=1,\ldots,m, \\
        & x_i\in\mathcal{X}_i,\ \ i=1,\ldots,n.
    \end{aligned} 
\label{problem:1}
\end{equation}
Here $f_0:\mathbb R^d \rightarrow \mathbb R$ denotes the global objective function, and each $f_i:\mathbb{R}^d\rightarrow\mathbb{R}$ represents the local cost function of agent $i$. The inequalities $\sum_{i=1}^n g_{ij}(x^i)\leq 0\,j=1,\ldots,m$ represent $m$ coupled constraints among 
$n$ agents, where $g_{ij}:\mathbb R^{d_i} \rightarrow \mathbb R$ for each $i=1,\ldots,n$ and $j=1,\ldots,m$. We also introduce the vector-valued function $g_i:\mathbb{R}^{d_i}\rightarrow\mathbb{R}^m$ defined as
\[
g_i(x^i) = \begin{bmatrix}
g_{i1}(x^i) \\ \vdots \\ g_{im}(x^i)
\end{bmatrix};
\]
with this notation, the $m$ coupled constraints can be concisely written as $\sum_{i=1}^n g_i(x^i)\leq 0$ with $\leq$ being interpreted component-wise. Note that the local cost function $f_i$ of each agent is influenced not only by its own action vector $x^i$, but by the joint action profile $x$. On the other hand, $g_i: \mathbb R^{d_i}\rightarrow \mathbb R^{m}$ for agent $i$ depends solely on its own action vector $x^i$.

In this paper, we study the setting where each agent $i$ can
only access zeroth-order function value of its local cost $f_i$ and constraint function $g_{ij}$, and
gradients of $f_i$ and $g_{ij}$ are not available. Furthermore, we impose the following mechanism of how each agent can access its associated zeroth-order information: First, each agent $i$ first determines its local action $x^i$. Then, after the agents take their local actions $x^1,\ldots,x^n$, each agent $i$ will receive a  corresponding local cost $f_i(x) = f_i(x^1, ..., x^n)$ together with the value of its local constraint function $g_i(x^i)$. In other words, the function values of $f_i$ and $g_i$ can only be obtained through observation of feedback values after actions have been taken. Similar settings without explicit coupled constraints have been considered by \cite{tang2023zeroth,zheng2023zeroth,marden2013model,menon2013convergence}, etc.

\begin{remark}
In the problem formulation~\eqref{problem:1}, we assume the coupled constraints take the form $\sum_{i=1}^n g_{ij}(x^i)\leq 0$. This form of coupled constraints is common and practically relevant, and has been investigated by many existing works including \cite{chang2014distributed,falsone2017dual,notarnicola2019constraint,liang2019distributed,wu2022distributed}, etc. %These works also highlight the significance of coupled constraints and suggest a continued need for methodologies that can adeptly address such issues.
Handling coupled constraints of more general forms is out of the scope of this paper, but will be an important future direction to explore.
\end{remark}

We now present some assumptions related to the problem.

% 1 凸性 f_0 与 g_i for all i = 1, ..., n
\begin{assumption}
\label{assumption_convex}
The functions $f_0$ and $g_{ij}$ for all $i,j$ are convex, and there exists $Z\geq 0$ such that $\|g_i(x)\|\leq Z$ for all $x\in\mathcal{X}_i$ and $i=1,\ldots,n$.

Furthermore, the problem~\eqref{problem:1} has an optimal primal-dual pair $(x^\ast,y^\ast)$.
\end{assumption}

% 2 f_0 满足 L_0 smooth 
% g_ij 满足 L_ij smooth
\begin{assumption}
\label{assumption_smooth}
Each function $f_i$ is $L_0$-smooth, and each function $g_{ij}$ is $L_{ij}$-smooth. In other words,
\[
\|\nabla f_i(x) - \nabla f_i(y)\| \leq L_0 \|x - y\|
\]
for all $x,y\in\mathbb{R}^d$, and
\[
\|\nabla g_{ij}(x) - \nabla g_{ij}(y)\| \leq L_{ij} \|x - y\|
\]
for all $x,y\in\mathbb{R}^{d_i}$. Additionally, define
\[
L_i \coloneqq \left(\sum\nolimits_{j=1}^m L_{ij}^2\right)^{\!\frac{1}{2}}
 \quad \text{and} \quad
 L_g \coloneqq \left(\sum\nolimits_{i=1}^n L_i^2\right)^{\!\frac{1}{2}}.
\]
\end{assumption}

% f_0 满足 M_0 lipschitz
% g_ij 满足 M_ij lipschitz
\begin{assumption}
\label{assumption_lipschitz}
Each function $f_i$ is $M_0$-Lipschitz continuous over $\mathcal{X}$, and each function $g_{ij}$ is $M_{ij}$-Lipschitz continuous over $\mathcal{X}_i$. In other words,
\[
|f_i(x) - f_i(y)| \leq M_0 \|x - y\|
\]
for all $x,y\in\mathcal{X}$, and 
\[
|g_{ij}(x) - g_{ij}(y)| \leq M_{ij} \|x - y\|
\]
for all $x,y\in\mathcal{X}_i$. Additionally, define
\[
M_i \coloneqq \left(\sum\nolimits_{j=1}^m M_{ij}^2\right)^{\!\frac{1}{2}} \quad \text{and} \quad M_g \coloneqq \left(\sum\nolimits_{i=1}^n M_i^2\right)^{\!\frac{1}{2}}.
\]
\end{assumption}

\section{Algorithm Design}

To effectively tackle optimization problems where direct access to derivatives is not feasible, we turn our attention to zeroth-order optimization. Particularly, we leverage the technique of zeroth-order gradient estimation, where one uses randomly explored function values to construct a stochastic gradient. Given a smooth function $f:\mathbb{R}^d\rightarrow\mathbb{R}$ and an arbitrary point $x\in\mathbb{R}^d$, a typical zeroth-order gradient estimator for $\nabla f(x)$ is given by
\begin{align*}
G_f(x, u, z) = 
\frac{f(x+uz)-f(x-uz)}{2u} z
.
\end{align*}
Here, $u$ is a positive real number called the smoothing radius; $z$ is a random perturbation vector. Notably, existing works \cite{nesterov2017random} have shown that when $z$ is sampled from the Gaussian distribution $\mathcal N(0, I_d)$, the expected value of this gradient estimator yields
\begin{align*}
\mathbb E_{z\sim\mathcal{N}(0,I_d)}[G_f(x, u, z)] = \nabla f^u(x),
\end{align*}
where $f^u(x) = \mathbb E_{y\sim\mathcal{N}(0,I_d)}[ f(x+uz) ]$ is a smoothed version of the original function $f$. This identity is the rationale behind using $G_f(x,u,z)$ as an estimator for the true gradient, even in the absence of derivative information.

Note that the main problem~\eqref{problem:1} contains coupled constraints of the form $\sum_{i=1}^n g_{ij}(x^i)\leq 0$. To handle such constraints, we consider the framework of primal-dual methods applied on the Lagrangian function. The Lagrangian function associated with~\eqref{problem:1} is given by
\begin{align*}
\mathcal L(x, y) = f_0(x) + 
\left\langle y, \sum_{i=1}^n g_i(x^i) \right\rangle.
\end{align*}
Here, $y\in \mathbb R^m$ servers as the Lagrange multiplier (i.e., the dual variable), which facilitates the incorporation of constraints into the algorithm design.

Our algorithm design consists of three core ingredients:

\vspace{3pt}
\noindent\textbf{1. Primal-dual gradient method with constraint extrapolation}: In this paper, we adapt the constraint extrapolation technique proposed in~\cite{boob2023stochastic} to our multi-agent feedback optimization setup, to effectively handle the coupled constraints. Specifically, at time step $t$, each agent first generates two independent perturbation variables, $z_t^i$ and $\hat z_t^i$, both drawn from the Gaussian distribution $\mathcal N(0, I_{d_i})$. We let $z_t\in\mathbb{R}^d$ denote the concatenation of $z_t^1,\ldots,z_t^n$, which follows the Gaussian distribution $\mathcal{N}(0,I_d)$. Then one constructs
\begin{equation}\label{eq:grad_est_f0}
\begin{aligned}
G_0^i(t) ={} & \frac{f_0(x_t+uz_t)-f_0(x_t-uz_t)}
{2u}z^i_t \\
={} &
\frac{1}{n}\sum_{j=1}^n \frac{f_j(x_t+uz_t)-f_j(x_t-uz_t)}{2u} z^i_t.
\end{aligned}
\end{equation}
This estimator essentially estimates the partial gradient of $f_0$ with respect to agent $i$'s local action. In contrast, the gradient estimator for the constraint function is more complex due to the nature of $g_i$ being a function that maps form $\mathbb R^{d_i}$ to $\mathbb R^m$:
\begin{align*}
G_{ij}(t) = \frac{g_{ij}(x_t^i+u\hat z^i_t)-g_{ij}(x^i_t-u\hat z^i_t)}
{2u}z^i_t
.
\end{align*}
Note that in constructing $G_{ij}(t)$, we employ the random perturbation $\hat{z}_t^i$ rather than $z_t^i$. We then assemble $G_{ij}(t),\,j=1,\ldots,m$ into a matrix representation
\begin{align*}
G_i(t) &= \begin{bmatrix}
G_{i1}(t)^T\\
\vdots\\
G_{im}(t)^T
\end{bmatrix} \in \mathbb R^{m\times d_i}
.
\end{align*}

Next, we introduce the quantity $\ell _G^i(t)$, which can be viewed as a linear approximation of $g_i$ at the estimate $x_t^i$:
\begin{align*}
\ell_{G}^i(t) = 
g_i(x^i_{t-1}) + G_i(t-1)(x^i_t - x^i_{t-1})
.
\end{align*}
%This linear approximation serves as a local model of the constraint function.
We then compute an ``extrapolated'' estimate $s_t^i$ that incorporates the effects of the previous estimates
\begin{align}
s_t^i = (1+\theta _t) 
\ell_{G}^i(t) - \theta _t \ell_{G}^i({t-1}).
\label{constraint_information}
\end{align}
where $\theta_t$ is a parameter that adjusts the influence of past information. This constraint extrapolation technique lies at the core of the algorithm proposed by~\cite{boob2023stochastic,nguyen2023stochastic}, and here we tailor it to our multi-agent zeroth-order feedback optimization setup. We shall later see that the choice $\theta_t=1$ suffices to ensure convergence of our algorithm. The quantity $\sum_{i=1}^n s_t^i$ will serve as the ``gradient'' of the Lagrangian with respect to the dual variable in the primal-dual framework.

Next, each agent $i$ generate yet another perturbation vector $\bar z_t^i \sim \mathcal N(0,I_d)$, and construct
\begin{align}
H_{ij}(t) = 
\frac{ g_{ij} (x^i_t + u \bar z_t^i ) - g_{ij}(x^i_t-u\bar z_t^i)  }{2u} \bar z_t^i.
\label{differential_information2}
\end{align}
It's not hard to see that $H_{ij}(t)$ share the same form with $G_{ij}(t)$ except that in~\ref{differential_information2} we employ $\bar{z}_t^i$ as the perturbation. The reason why we introduce three sets of perturbations is rather technical; we only remark that the independence of $z_t^i,\hat{z}_t^i,\bar{z}_t^i$ will be critical for convergence analysis.

With these components in place, the prototype of our proposed algorithm is given by
\begin{align}
y_{t+1} & = \argmin_{y\ge 0, \|y\| \le C   } \!\left\{-\left\langle
\sum_{i=1}^n s^i_t, y \right\rangle + \frac {1}{2\mu_t}\|y-y_t\|_2^2
\right\}\!,
\label{eq:prototype_dual}\\
x_{t+1}^i  & = \argmin_{x\in \mathcal X_i} \Big\{\langle V^i_t, x \rangle + \frac{1}{2\eta_t} \|x - x_t^i\|^2 \Big\},
\label{eq:prototype_primal}
\end{align}
where we denote
\[
V^i_t = G_0^i(t)+ \sum_{j=1}^m H_{ij}(t)\cdot [y_{t+1}]_j.
\]
Here $\eta_t$ and $\mu_t$ are the step sizes of the primal step and the dual step, respectively. $C\geq 0$ is a sufficiently large constant that bounds the norm of the optimal dual variable. $[y_{t+1}]_j$ denotes the $j$th component of the $m$-dimensional vector $y_{t+1}$.

However, the aforementioned algorithm prototype still faces two significant issues in our distributed scenarios:
\begin{itemize}[leftmargin=15pt, labelwidth=9pt,labelsep=6pt]
\item In constructing the partial gradient estimator~\eqref{eq:grad_est_f0}, each agent $i$ still needs to know the difference information $f_j(x_t+uz_t)-f_j(x_t-uz_t)$ from other agents $j\neq i$. We need to design an inter-agent communication protocol to address this issue.

\item The dual iterate $y_t$ in~\eqref{eq:prototype_dual} is still a global variable whose update requires aggregating all $s_t^i$, which cannot be directly implemented in our distributed scenario.
\end{itemize}

To overcome these two issues, we introduce the next two core ingredients in our algorithm design.

\vspace{5pt}
\noindent\textbf{2. Protocol for exchanging difference information between agents:} 
To facilitate effective information exchange between agents, we draw upon the approach outlined in \cite{tang2023zeroth}. Specifically, for each agent $i$, we maintain an information array structured as follows
\begin{equation}\label{eq:difference_array}
\begin{array}{|c|c|c|c|}
\hline
{D^i_1(t)} & D^i_2(t) & \cdots & D^i_n(t)
\\ \hline
\tau^i_1(t) & \tau^i_2(t) & \cdots 
& \tau^i_n(t)
\\ \hline
\end{array}
\end{equation}
In this array, $D^i_j(t)$ represents the difference of $f_j$ collected from agent $j$, while $\tau^i_j(t)$ denotes the time step at which $D^i_j(t)$ was collected by agent $j$. The pair $(D^i_j(t), \tau^i_j(t))$ 
represents the latest difference information available to agent $i$ from agent $j$ at time $\tau _j^i(t)$. The items in this array are update according to the following rules:
\begin{enumerate}[leftmargin=0pt, itemindent=15pt, listparindent=15pt,labelsep=7pt]
\item At each time step $t$, each agent $i$ collects its own difference information as follows
\begin{align}
D^i_i(t) = \frac{f_i(x_t+uz_t) - f_i(x_t - uz_t)}{2u} z_t^i.
\label{differential_information}
\end{align}
This quantity can be obtained by letting each agent $i$ apply the local actions $x_t^i\pm uz_t^i$ simultaneously and observing the associated local costs. Each agent $i$ also records $\tau_i^i(t)=t$.

\item Each agent $i$
retrieves the information arrays of its neighboring agents via the communication network. We let $(D^{k\rightarrow i}_j(t), \tau^{k\rightarrow i}_j(t))$ denote the $j$th column of agent $k$th array~\eqref{eq:difference_array} received by agent $i$ at time $t$.

\item Then, for each $j\neq i$, agent $i$ identifies the neighbor $k$ such that
$\tau^{k\rightarrow i}_j(t)$ has the maximum value, which corresponds to the most up-to-date difference information of $f_j$. We then update $D^i_j(t)$ along with $\tau^i_j(t)$ to accord with this most up-to-date information.
\end{enumerate}

The above process allows agent $i$ to update its own array with the most recent difference information available from its peers. Particularly, we have
$
D^i_j(t) = D^j_j(\tau^i_j(t))
$, and $\tau^i_j(t)=t-b_{ij}$ if no error occurs during communication; see~\cite{tang2023zeroth}.

Finally, we obtain a partial gradient estimator with delayed information for agent $i$, given by
\begin{align}
\label{eq:modified_f0_grad_est}
G_0^i(t) = \frac{1}{n} \sum_{j=1}^n D^i_j(t)\,z^i_{\tau^i_j(t)},
\end{align}
which replaces the original partial gradient estimator in~\eqref{eq:grad_est_f0}.
%This addresses the issue where the gradient estimator directly produced by agent $i$ does not include differential information from other agents.

\vspace{5pt}
\noindent\textbf{3. Averaging consensus for the dual variable $y_t$: }The averaging consensus method is a fundamental approach used in distributed systems to enable multiple agents to agree on a common value. %This method is particularly useful in scenarios where agents operate in a networked environment and need to collaborate without centralized control.
%In the averaging consensus framework, agents leverages local communications with their neighbors, and incorporates the received values to compute an updated state. 
We will apply the consensus technique in our context as follows. For $i = 1,\ldots, n$, each agent $i$ will keep a local copy $y_t^i$ that serves as an estimate of the global dual variable $y_t$. At each time step $t$, agent $i$ sends its current value $y^i_t$ to all its 
neighbors $j$ (i.e., $(i,j)\in \mathcal E$), and also receives the values $y_t^j$ from its neighbors. Agent $i$ then compute the weighted average
\begin{align}\label{consensus_dual}
p^i_t = \sum_{j=1}^n W_{ij}y^i_t.
\end{align}
where $W=[W_{ij}]\in\mathbb{R}^{n\times n}$ is a double stochastic matrix and $W_{ij}$ represents the weight associated with the connection between agents $i$ and $j$. We assume $W_{ii}>0$ for each $i$ and $W_{ij}=0$ if no communication link exists between agents $i$ and $j$. We then modify the dual update step as
\begin{equation}
y_{t+1}^i = \argmin_{y\ge 0, \|y\| \le C   } \!\left\{-\left\langle
s^i_t, y \right\rangle + \frac {1}{2\mu_t}\|y-p_t^i\|_2^2
\right\}\!.
\label{eq:modified_dual}
\end{equation}
It can be shown that this averaging consensus technique ensures all agents' $y_{t}^i$ will converge towards a common value.

\begin{algorithm}[!tbh]
\caption{Multi-Agent Zeroth-order Feedback Optimization with Coupled Constraints (\texttt{MAZFO-CoupledCon})}
\label{alg:main}
\begin{algorithmic}[1]
    
\State \textbf{Parameters:} Step sizes $\eta_t,\mu_t$, extrapolation parameter $\theta_t$, smoothing radius $u$, communication matrix $W$, weights $\gamma_t$.

\For{$t = 1, ..., T-1$ }
    \State Agent $i$ generates $z_t^i\sim \mathcal N(0, I_d)$, and updates $D_i^i(t)$ via~\eqref{differential_information} and sets $\tau_i^i(t) = t$.
    \State Agent $i$ generates  $\hat z_t^i \sim \mathcal N(0, I_d)$, and constructs $s_t^i$ via~\eqref{constraint_information}.
    \State Agent $i$ receives from its neighbor $k$ the difference information $(D^{k\rightarrow i}_j(t),\tau^{k\rightarrow i}_j(t))$ for each $j\neq i$ and $k:(k,i)\in\mathcal{E}$, and updates
    \begin{align*}
    k^i_j(t) = \argmax_{k:(k, i)\in \mathcal E} \ \tau^{k\rightarrow i}_j(t)
    \end{align*}
    and $
    \tau_j^i(t) = \tau_j^{k_j^i(t)\rightarrow i}(t),\ 
    D_j^i(t) = D_j^{k_j^i(t)\rightarrow i}(t)$
    for $j\ne i$.
    \State Agent $i$ performs the averaging consensus step
    \begin{align}
    p^i(t) = \sum_{j=1}^n W_{ij}y^j(t).
    \tag{\ref{consensus_dual}}
    \end{align}
    \State Agent $i$ performs dual iteration
    \begin{align}
        y^i_{t+1} = \argmin_{y\ge 0, \|y\| \le C   } \Big\{ \langle -s^i_t, y \rangle + \frac {1}{2\mu_t}\|y-p^i_t\|_2^2 \big\}.
    \tag{\ref{eq:modified_dual}}
    \end{align}
    \State Agent $i$ generates $\bar z \sim \mathcal N(0, I_d)$ and constructs $H_{ij}(t)$ via \eqref{differential_information2}.
    \State Agent $i$ constructs the partial gradient estimator
    \begin{align}
        G_0^i(t) = \frac{1}{n} \sum_{j=1}^n D^i_j(t) z^i_{\tau^i_j(t)}.
        \tag{\ref{eq:modified_f0_grad_est}}
    \end{align}
    \State Agent $i$ updates
    \begin{align}
    V^i_t &= G_0^i(t)+ \sum _{j=1}^m H_{ij}(t)\cdot [y^{i}_{t+1}]_j
    \\
    x^i_{t+1}  &= \arg \min _{x\in \mathcal X_i} \Big\{\langle V^i_t, x \rangle + \frac{1}{2\eta_t} \mathcal \|x- x_t^i\|^2 \Big\} \label{eq:primal_update}
    .
    \end{align}
    \label{algorithm:1}
\EndFor

\State \textbf{Return:} $\bar x_T 
= \sum_{t=0}^{T-1} \gamma_t x_{t+1} / \sum_{t=0}^{T-1}\gamma_t $.

\end{algorithmic}
\end{algorithm}

By summarizing the aforementioned ingredients, we obtain our proposed algorithm for multi-agent zeroth-order feedback optimization with coupled constraints, outlined in Algorithm~\ref{alg:main}. Note that the output of Algorithm~\ref{alg:main} is a weighted average $\bar{x}_T = \sum_{t=0}^{T-1} \gamma_t x_{t+1} / \sum_{t=0}^{T-1}\gamma_t$; in the next section, we will present one way to set the weights $\gamma_t$ that enjoys convergence guarantees.

\begin{remark}
Note that Step 7 of Algorithm~\ref{algorithm:1} still requires projecting the dual variable back to the bounded region $\{y:\|y\|\leq C\}$, which seems common for primal-dual methods but is avoided by the original first-order constraint extrapolation method in~\cite{boob2023stochastic}. We mention that this limitation is due to a technical point related to the bias of zeroth-order gradient estimation. We leave it as a future work to study how to eliminate the projection onto $\{y:\|y\|\leq C\}$ in the algorithm design.
\end{remark}

\section{Convergence Analysis}

In this section, we present convergence analysis results of our proposed algorithm. The detailed proofs of these results will be postponed to the Supplementary Materials of this paper.

We first introduce some auxiliary quantities. We let $\bar{R}_i \coloneqq \sup_{x\in\mathcal{X}_i} \|x\|$ and $\bar R \coloneqq \left(\sum_{i=1}^n \bar{R}_i^2\right)^{1/2}$; these quantities characterize the size of the feasible region for the primal variable. Then, we introduce $\bar b$ and $\bar{\mathfrak b}$ to quantify the connectivity of network, defined as
\begin{align*}
\bar b \coloneqq \left(\frac{\sum_{i,j=1}^n b_{ij}^2}{n^2} \right)^{\!\frac{1}{2}},
\qquad
\bar{\mathfrak b}
\coloneqq 
\left( 
\frac{\sum_{i,j=1}^nb_{ij}^2 d_i }{nd}
\right)^{\!\frac{1}{2}}
.
\end{align*}
Intuitively, a smaller value of $\bar b$ or $\bar{\mathfrak b}$ corresponds to a higher concentration of nodes within the graph $\mathcal G$. We also define
\[
\rho = \left\|W - \frac{1}{n}\mathbf{1}\mathbf{1}^T\right\|
< 1.
\]

Our main results are summarized in the following theorem.

\begin{theorem}
\label{theorem:1}
Suppose Assumptions \ref{assumption_convex}, \ref{assumption_smooth} and \ref{assumption_lipschitz} hold. 
Let $y_0^i = 0$ and $x_0^i = 0$, and set the parameters of Algorithm 1 as follows: $\theta_t=1$, $\gamma_t = 1$, $\frac{1}{\eta_t} = L_0 + L_{\max} + \frac{1}{\eta}$ and $\mu_t = \mu$ with
\[
\begin{aligned}
& {\eta} = \frac{\bar R}{\sqrt {T \xi}},
\qquad
{\mu} = \frac{C\sqrt {2n}}{\sqrt {T\zeta}},
\qquad u = \min\!\left\{ \frac{M_g}{(d+6)L_g}, 
\frac{1}{
(d\sqrt T \max \{L_0, L_{g} \}) ^{\frac{1}{2}}}
\right\},
\end{aligned}
\]
where
\[
\begin{aligned}
\xi ={} &
(M_0\bar {\mathfrak b} \sqrt d + L_0\bar b d \bar R + 2\sqrt 3 \bar {\mathfrak b} d M_0)(24M_0^2 + 27M_g^2 C^2)^{\frac{1}{2}}
+ 104M_0^2d + 124M_g^2 dC^2, \\
\zeta ={} &
 403d M_g^2\bar R + \frac{1}{1-\rho}(
6dZ^2 + 3M_g^2 \bar R + 243 \bar RdM_g^2
).
\end{aligned}
\]
Then, we have
\begin{equation}
\begin{aligned}
\label{result_convergence_rate}
\mathbb E[f_0(\bar x_T) - f_0(x^*)]
\le{}
&\frac{1}{\sqrt{T}}\left(
2 \bar R \sqrt {\xi} + 
C\sqrt {n\zeta} + \frac{2C+1}{2}
\right) \\
&
+ \frac{(L_0+L_{\max})\bar R^2}{T}
+\frac{7C\sqrt {2n}}{T^{\frac{3}{2}} \sqrt \zeta}
,
\end{aligned}
\end{equation}
and
\begin{equation}
\begin{aligned}
\label{result_violation}
\mathbb E \!\left[
\left\|\left[\sum_{i=1}^n g_i(\bar x^i_T)\right]_+\right\|\right] 
\le{}
&
\frac{1}{\sqrt T}\left(
\bar R \sqrt \xi + \bar R \frac{(\xi + L_g\bar R C)}{\sqrt \xi}
+ 2 C\sqrt {n\zeta} + %\frac{C}{2}
C
\right) \\
&
+\frac{(L_0+L_{\max}) \bar R^2}{T}
+ \frac{7 C\sqrt {2n}}{T^{\frac{3}{2}} \sqrt {\zeta} },
\end{aligned}
\end{equation}
where $\bar{x}_T = \sum_{t=0}^{T-1} \gamma_t x_{t+1} / \sum_{t=0}^{T-1}\gamma_t$, and $[\cdot]_+$ denotes taking the positive part component-wise of a vector.
\end{theorem}

\begin{corollary}
Under Assumptions \ref{assumption_convex}, \ref{assumption_smooth} and \ref{assumption_lipschitz}, when the parameters of Algorithm~\ref{alg:main} are properly chosen, the number of zeroth-order queries per agent to achieve $\mathbb E[f_0(\bar x_T) - f_0(x^*)]\leq\epsilon$ and $\mathbb E \!\left[
\left\|\left[\sum_{i=1}^n g_i(\bar x^i_T)\right]_+\right\|\right]\leq\epsilon$ can be bounded by
\begin{equation}\label{eq:oracle_complexity}
O\!\left(\frac{d}{\epsilon^2}\max\left\{
\bar b+\bar{\mathfrak b},
\frac{n}{1-\rho}
\right\}\right).
\end{equation}
\end{corollary}

We can see that the oracle complexity of Algorithm~\ref{alg:main} has a $O(\epsilon^{-2})$ dependence on the optimization accuracy $\epsilon$, which is typical for stochastic first-order methods. The dimensional dependence is $O(d)$, which matches the typical results of zeroth-order optimization with two-point gradient estimation. The bound~\eqref{eq:oracle_complexity} also explicitly demonstrates the dependence of the oracle complexity on the network's topological characterizations $\bar b, \bar{\mathfrak b},\rho$ as well as the number of agents $n$, which reveals the scalability of our proposed algorithm.

\begin{remark}
In Theorem~\ref{theorem:1}, we employ constant algorithmic parameters $\eta_t,\mu_t$, etc. that depend on the total number of iterations $T$ planned in advance. Such analysis paradigm is prevalent in the study of stochastic first-order methods, allowing simpler proofs while still effectively providing useful oracle complexity results. We believe that Algorithm~\ref{alg:main} can also achieve convergence if we run the iterations indefinitely and employ diminishing step sizes, but detailed analysis seems tedious and is out of the scope of this paper.
\end{remark}

\section{Numerical Experiments}
\label{section:numerical}
% 实验部分

In this section, we conduct preliminary numerical experiments to test the performance of Algorithm~\ref{alg:main}. The test case consists of $n=15$ agents, with the total dimension being $d=40$. The local objective functions of the test case are quadratic functions of the form
\[
f_i(x) = x^T A_i x + b_i^T x + c_i,
\]
where $A_i\in\mathbb{R}^{d\times d}$ is positive definite, $b_i\in\mathbb{R}^d$ and $c_i\in\mathbb{R}$. The global objective function is then
$
f_0(x) = x^T A x + b^T x + c$ where $A=\frac{1}{n}\sum_{i=1}^n A_i$, $b=\frac{1}{n}\sum_{i=1}^n b_i$, $c=\frac{1}{n}\sum_{i=1}^n c_i$. We randomly generate one instance of $A_i,b_i,c_i$ so that $A$ has eigenvalues within $[0.1,1.6]$. The test case has two inequality constraints
$
\sum_{i=1}^n g_{ij}(x^i)\leq 0,\ i=1,2$ where
\[
g_{ij}(x^i)=
(x^i)^T P_{ij} x^i + q_{ij}^T x^i + r_{ij}.
\]
Each $P_{ij}\in\mathbb{R}^{d_i\times d_i}$ is also randomly generated and positive definite with eigenvalues in $[0.1,1.6]$. The optimal value of this test case is $f^\ast = -8.9368$. For the numerical experiments, we run $100$ random trials of our algorithm for each choice of algorithmic parameters. We set the smoothing radius to be $u=0.01$.

\begin{figure}
\centering
\includegraphics[width=.33\linewidth]{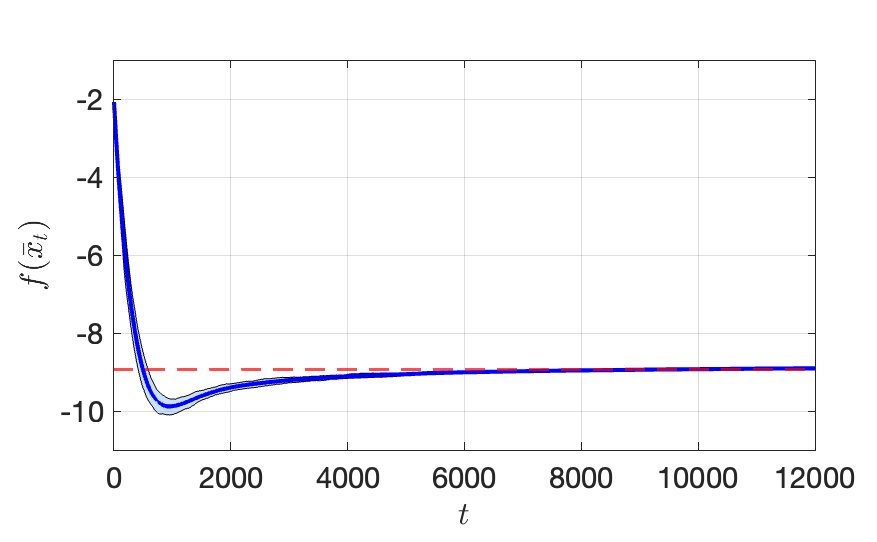} \hfil
\includegraphics[width=.33\linewidth]{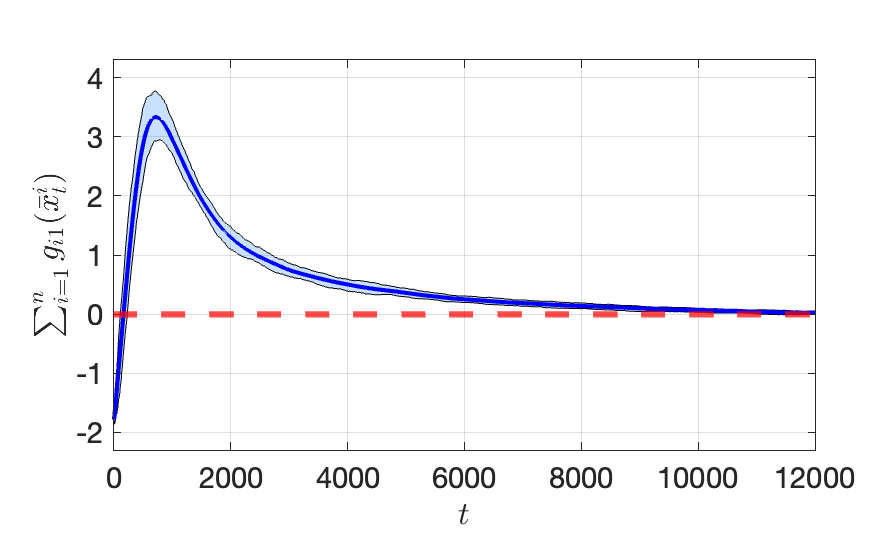}\hfil
\includegraphics[width=.33\linewidth]{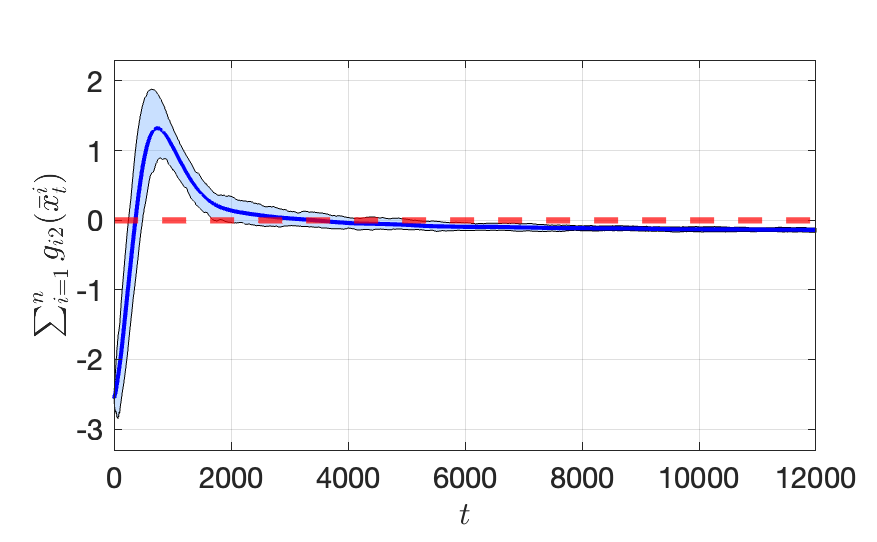}
\caption{Convergence of Algorithm~\ref{alg:main} on the numerical test case with constant step sizes $\eta_t=\mu_t=1/500$.}
\label{fig:const_small}
\end{figure}

Figure~\ref{fig:const_small} shows the convergence of Algorithm~\ref{alg:main} with constant step sizes $\eta_t=\mu_t=1/500$. We plot the global function value $f(\bar x_t)$ as well as the constraint function values $\sum_{i} g_{ij}(\bar{x}_t^i)$, where 
$\bar{x}_t= \frac{1}{t}\sum_{\tau=1}^t x_\tau$ and $\bar{x}_t^i = \frac{1}{t}\sum_{\tau=1}^t x_\tau^i$. The dark blue curve represents the average of the $100$ random trials, while the light blue shade represents the the corresponding light shade represents the 5\% to 95\% quantile interval among these trials. We can see that, with proper constant step sizes, Algorithm~\ref{alg:main} is able to converge to a solution with small optimality gap and constraint violations that is consistent with our theoretical result.

\begin{figure}
\centering
\includegraphics[width=.33\linewidth]{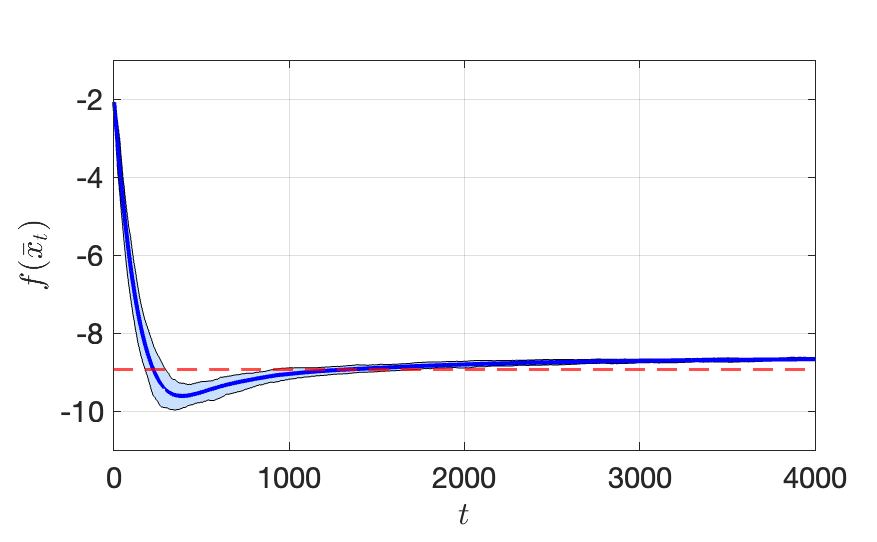} \hfil
\includegraphics[width=.33\linewidth]{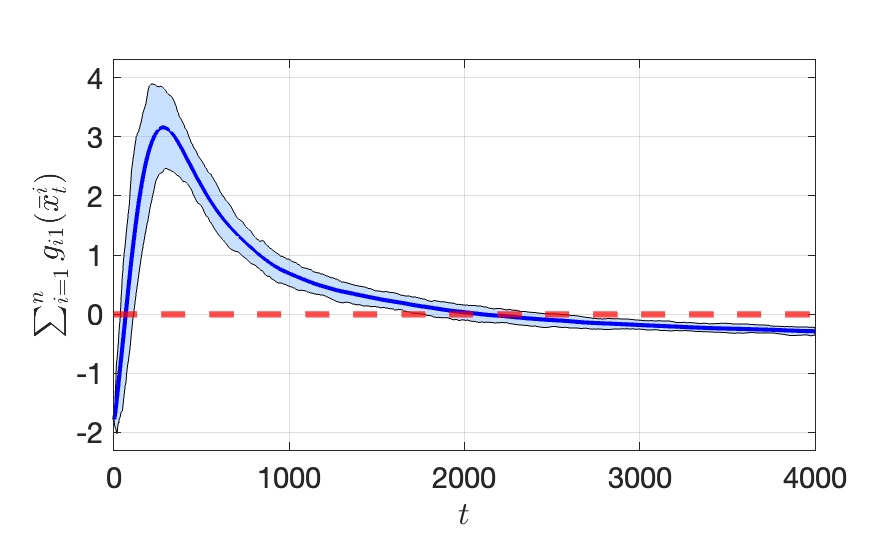}\hfil
\includegraphics[width=.33\linewidth]{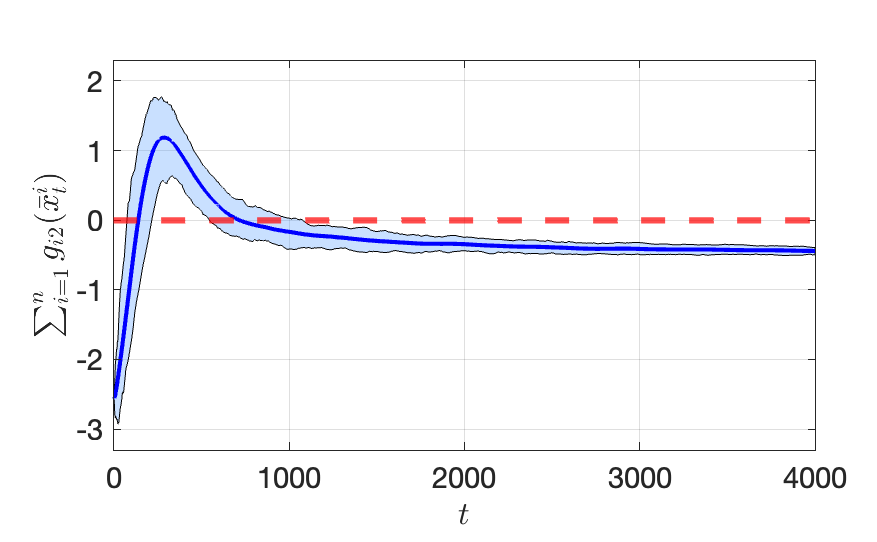}
\caption{Convergence of Algorithm~\ref{alg:main} on the numerical test case with constant step sizes $\eta_t=\mu_t=1/200$.}
\label{fig:const_big}
\end{figure}

Figure~\ref{fig:const_big} shows the convergence of Algorithm~\ref{alg:main} with larger constant step sizes $\eta_t=\mu_t=1/200$. It is noteworthy that while the large step size facilitated faster convergence, it still exhibited a larger gap from the optimal value compared to Figure~\ref{fig:const_small}. In contrast, the smaller step size led to better accuracy but required an excessive number of iterations, resulting in slower convergence of the algorithm. We point out that these phenomena are typical for stochastic-gradient-descent-type algorithms with constant step sizes, and are also consistent with our theoretical result.

\begin{figure}
\centering
\includegraphics[width=.33\linewidth]{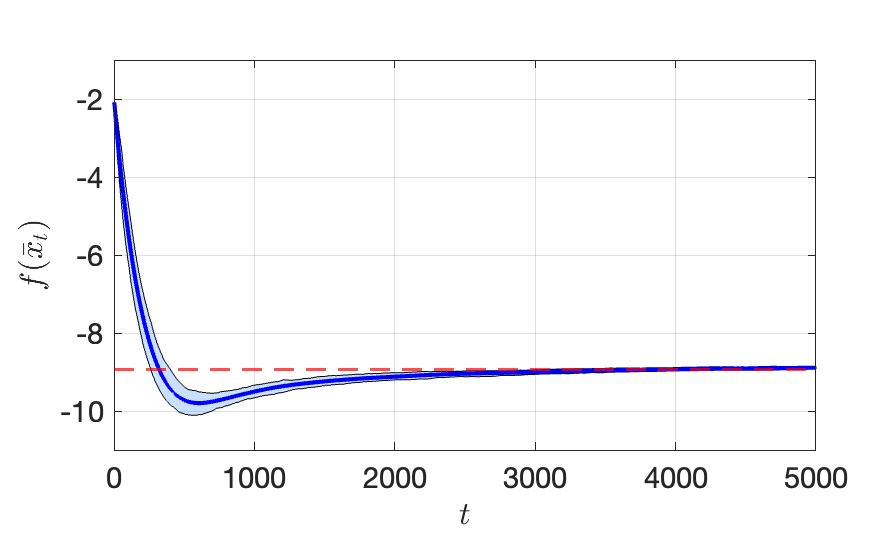}\hfil
\includegraphics[width=.33\linewidth]{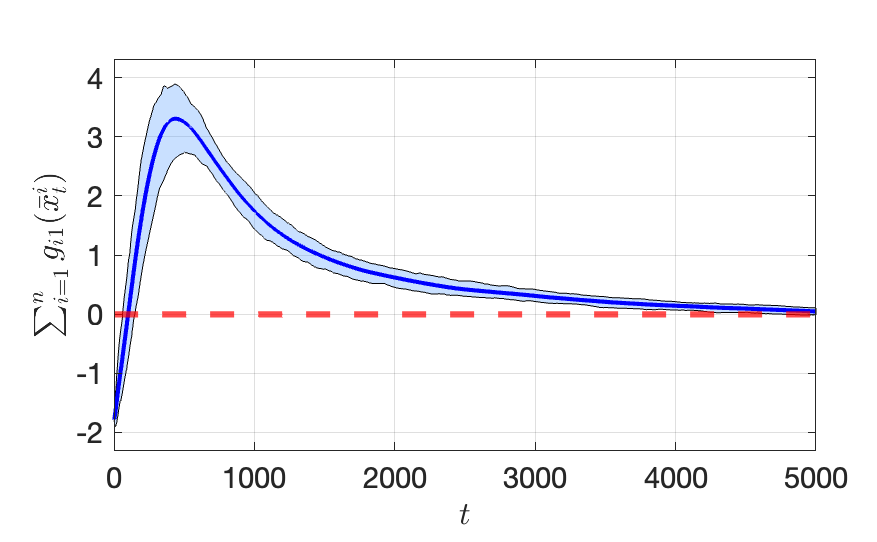}\hfil
\includegraphics[width=.33\linewidth]{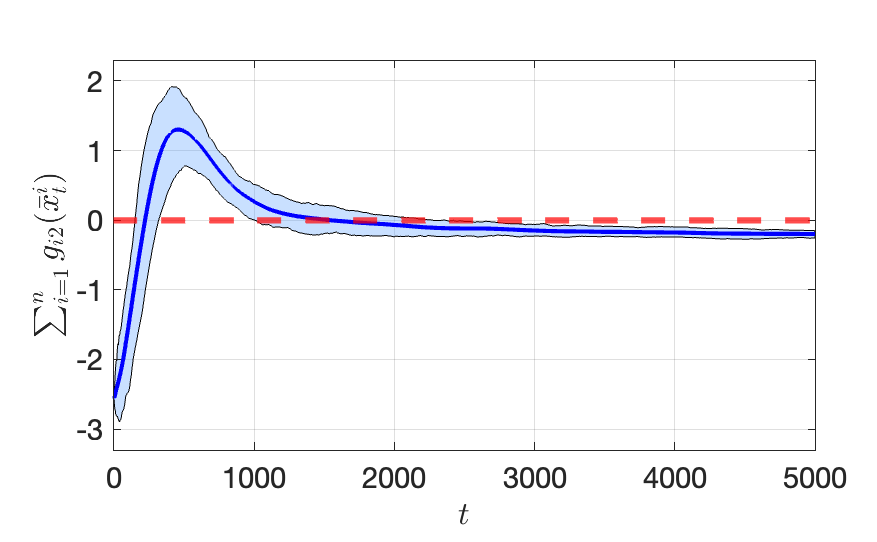}
\caption{Convergence of Algorithm~\ref{alg:main} on the numerical test case with diminishing step sizes $\eta_t=\mu_t=1/(\sqrt{t}+300)$.}
\label{fig:dim_step}
\end{figure}

We have also tested the algorithm with diminishing step sizes $\eta_t=\mu_t=1/(\sqrt{t}+300)$, and the results are shown in Figure~\ref{fig:dim_step}. We can see that the diminishing step size setting demonstrated an initial convergence rate comparable to that of the large step size, while ultimately achieving superior convergence accuracy. This observation suggests that in practice, one may use properly chosen diminishing step sizes to for better convergence behavior; it also indicates that the diminishing step size setting could be a promising topic for future investigation.

\section{Conclusion}
% 结论部分
In this study, we proposed a distributed zeroth-order feedback optimization algorithm specifically designed for cooperative multi-agent systems facing coupled constraints of the form $\sum_i g_{ij}(x^i)\leq 0$. Our approach utilizes constraint extrapolation techniques and the averaging consensus framework to effectively tackle the challenges posed by coupled constraints in decentralized settings. Additionally, we provided theoretical results on its convergence rate and oracle complexity. Numerical experiments were conducted, revealing that employing a diminishing step size may yield even better performance. %, characterized by faster convergence and improved accuracy.
Some potential directions for future research include extension to coupled constraints of more general forms, removing projection of the dual iterate onto the bounded set $\{y:\|y\|\leq C\}$, analysis of the algorithm with diminishing step sizes, etc.

\bibliographystyle{IEEEtran}
\bibliography{refs.bib}

\newpage
\include{supp.tex}

\end{document}

%% file: supp.tex
% 此为二范数版本证明材料
\section{Supplementary Materials}

In this section, we provide detailed proofs for our main result Theorem~\ref{theorem:1}.

\subsection{Some Auxiliary Lemmas}

\begin{lemma}
\label{lemma:3}
Suppose $h:\mathbb R^d \rightarrow \mathbb R$ 
is an $L$-smooth and $M$-Lipschitz continuous  function. Then we have
\begin{equation}
 | h^u(x) - h(x) | \le \min\left\{u M \sqrt d , \frac{1}{2}u^2 L d\right\},
\label{proof19}
\end{equation}
where $h^u(x) = 
\mathbb E_{z\sim\mathcal N(0, I_d)}[h(x+uz)]$.
\end{lemma}
\begin{proof}
Since $h$ is an $L-$smooth function, we have
\begin{equation*}
|h^u(x) - h(x)|
= |\mathbb E_{z\sim\mathcal N(0, I_d)}[ h(x+uz) - h(x) ] |
\le
\left| \mathbb E_{z\sim \mathcal N(0, I_d)}
[\langle \nabla h(x), uz \rangle + \frac{1}{2} 
\|uz \|^2]\right|
= \frac{1}{2} u^2 L d,
\end{equation*}
and
\begin{equation*}
|h^u(x) - h(x)| = 
|\mathbb E_{z\sim \mathcal N(0, I_d)} [ h(x+uz) - h(x)]|
\le uM\mathbb E_{z\sim \mathcal N(0, I_d)} [\|z\|] 
= uM\sqrt d.
\qedhere
\end{equation*}
\end{proof}

\begin{lemma}
\label{lemma:4}
Consider the following optimization problem
\begin{equation}
    \begin{aligned}
        \mathrm{minimize} & \quad f_0^u(x) = \sum_{i=1}^n\frac{1}{n}f_i^u(x^1, ..., x^n)
        \\
        \mathrm{subject\ to} & \quad 
        \sum_{i=1} ^ n g_i^u(x^i) =\sum_{i=1}^n
        \begin{pmatrix}
            g_{i1}^u(x^i)\\
            g_{i2}^u(x^i)\\
            \cdots \\
            g_{im}^u(x^i)
        \end{pmatrix}
        \preceq 0_m,
    \end{aligned} 
\label{problem:2}
\end{equation}
where $f_i^u = \mathbb{E}_{z \sim \mathcal{N}(0, I_d)} [f_i(x + uz)]$ and $g_{ij}^u = \mathbb{E}_{z \sim \mathcal{N}(0, I_d)} [g_{ij}(x + uz)]$, for all $i = 1, \ldots, n$ and $j = 1, \ldots, m$. Let an optimal primal-dual pair be denoted as $(x^{*,u}, y^{*,u})$, then we have
\begin{equation}
\begin{aligned}
f_0(x^*) \le f_0^u(x^{*, u}) \le f_0^u(x^*)
.
\end{aligned}
\label{proof20}
\end{equation}
\end{lemma}

\begin{proof}
Given the objective function and the constraints, we can establish the following inequality chain
\begin{equation*}
\begin{aligned}
f_0(x^*_1, ..., x_n^*) + \langle y^*, \sum g_i(x^*_i) \rangle&\le f_0(x^{*, u}) + \langle y^*, \sum g_i(x^{*, u}) \rangle
\\
&\le f_0^u(x^{*, u}) + \langle y^*, \sum g^u_i(x^{*, u}) \rangle
\\
&\le f_0^u(x^{*, u}) + \langle y^{*, u}, \sum g^u_i(x^{*, u}) \rangle.
\end{aligned}
\end{equation*}
By the complementary slackness condition, we have (\ref{proof20}).
\end{proof}

\subsection{A Critical Lemma for Establishing Convergence}
\begin{lemma}
\label{lemma:2}
Suppose $\{\gamma_t, \eta_t, \mu_t, \theta_t\}$ is a non-negative sequence that satisfying
\begin{equation}
\begin{aligned}
\gamma_t\theta_t = \gamma_{t-1}, \quad 
\frac{\gamma_t}{\mu_t} = \frac{\gamma_{t-1}}{\mu_{t-1}}
\end{aligned}
\label{gamma_and_theta_relation}
\end{equation}
and
\begin{align}
4M_i^2\frac{\theta_t}{\theta_{t-1}} &\le  \frac{ 
\frac{1}{\eta_{t-2}} - L_0 - L_i
}{12\mu_t}, \quad
M_i^2\theta_t \le   \frac{ \frac{1}{\eta_{t-1}} - L_0 - L_i}{12\mu_t},
\label{M_gamma_theta_ineq1}
\\
\frac{4M_i^2}{\theta_{T-1}} &\le \frac{\frac{1}{\eta_{T-2}} - L_0 - L_i }{12\mu_{T-1}}  , \quad
M_i^2 \le \frac{ \frac{1}{\eta_{T-1}} - L_0 - L_i  }{12\mu_{T-1}}.
\label{M_gamma_theta_ineq2}
\end{align}
Then for arbitrary $x\in\mathbb{R}^d$ and $y\in\mathbb{R}^m$, we have
\begin{equation}
\begin{aligned}
& \sum_{t = 0}^{T-1} \gamma_t \Big[ 
f^u_0(x_{t+1}) - f^u_0(x) + 
\sum_{i=1}^n\langle g^u_i(x^i_{t+1}), y\rangle  - 
\sum_{i=1}^n\langle g^u_i(x^i), y_{t+1}^i\rangle 
\Big]
\\
& + 
\sum_{t=0}^{T-1}\gamma_t  \Big[ \sum_{i=1}^n \langle \delta_t^{G_i}, x^i_{t}-x^i\rangle-
\sum_{i=1}^n
\langle \delta_{t+1}^{F_i}, y^i_{t+1} - y \rangle \Big]
\\
\le {} & 
\frac{\gamma_0}{2\eta_0} \|x-x_0\|_2^2 - \frac{\gamma_{T-1}}{2\eta_{T-1}} 
\|x-x_T\|_2^2
+ 
\sum_{i=1}^n \Big[
\frac{\gamma_0}{2\mu_0}\|y-y_0^i\|^2 - \frac{\gamma_{T-1}}{12\mu_{T-1}}\|y-y_T^i\|^2 \Big]
\\
& +  
\sum_{i=1}^n \Bigg[
\sum_{t= 0}^{T-1} \frac{2\gamma_t}{\frac{1}{\eta_t} - L_0 - L_i} \Big[ 
||\delta_t^{G_i}||^2+
\big( \frac{L_i\bar R_i}{2}[\|y\|-1]_+   \big)^2
\Big] 
+
\frac{3\gamma_{T-1}\mu_{T-1}}{2} \|q_T^i - \bar q_T^i\|^2
\\
&+
\sum_{t=1}^{T-1}  \frac{3\gamma_t\theta_t^2\mu_t}{2}\|q_t^i-\bar q_t^i\|^2 \Bigg]
+ \sum_{i=1}^n\sum_{t=0}^{T-1} \gamma_t\theta_t\langle q_t^i, p_t^i - y_t^i \rangle.
\end{aligned}
\label{proof18}
\end{equation}
Here $q_t^i = \ell _{G}^i(t) - \ell _{G}^i({t-1})$ and 
$\bar q_t^i = \ell_{g^u}^i(t) - \ell_{g^u}^i({t-1})$ in which we denote
\[
\ell_{g^u}^i(t) = g_i^u(x_{t-1}^i) + \nabla g_{i}^u(x_{t-1}^i)(x_{t}^i - x_{t-1}^i),
\]
and
\[
\delta_t^{G_i} = G_0^i(t) - \frac{\partial f_0^u}{\partial x^i}(x_t) + \sum_{j=1}^m[H_{ij}(t) - \nabla g_{ij}^u(x_t^i)] [y_{t+1}^i]_j,
\quad \delta_t^{F_i} = \ell_{G}^i(t) - \ell _{g^u}^i(t)
\]
for all $i = 1, \ldots, n$.
\end{lemma}

\begin{proof}
Considering the first-order optimality condition of (\ref{eq:modified_dual}), we have the following
\begin{equation*}
\begin{aligned}
\left\langle -s_t^i + \frac{1}{\mu_t}(y_{t+1}^i - p_t^i), y-y_{t+1}^i \right\rangle\ge 0.
\end{aligned}
\end{equation*}
This implies that
\begin{align}
\langle - s_t^i, y_{t+1}^i - y   \rangle \le 
\frac{1}{\mu_t} \langle 
y_{t+1}^i-p_t^i, y-y_{t+1}^i
\rangle
=\frac{1}{2\mu_t}\Big[
\|y-p_t^i \|_2^2 - 
\|y_{t+1}^i - p_t^i\|_2^2
-\|y-y_{t+1}\|_2^2
\Big]
.
\label{proof1}
\end{align}
Similarly, for (\ref{eq:primal_update}), utilizing its first-order optimality condition, we obtain
\begin{align}
\langle V^i_t , x^i_{t+1} - x^i \rangle  \le  
\frac{1}{\eta_t}\langle x_{t+1}-x_t^i, x- x_{t+1}^i \rangle
=\frac{1}{2\eta_t} \Big[
\|x-x_t^i\|_2^2 - 
\|x_{t+1}^i-x_t^i\|_2^2
- \|x-x_{t+1}^i\|_2^2
\Big]
.
\label{proof2}
\end{align}
Denote $v^i_t = \frac{\partial f^u_0}{\partial x^i}(x_t)    + \sum_{j=1}^m \nabla g^u_{ij}(x_t^i)[y_{t+1}^i] _j  $.  Due to the convexity of $g_i$, we have
\begin{equation}
\begin{aligned}
&\langle  v_t^i , x^i_{t+1} -x^i  \rangle 
\\
= {} &
\left\langle \frac{\partial f^u_0}{\partial x^i}(x_t),   x^i_{t+1} - x^i  \right\rangle  
+ \sum_{j=1}^m [y^i_{t+1}]_j \langle \nabla g^u_{ij}(x_t^i), x_{t+1}^i - x_t^i  \rangle 
+ \sum_{j=1}^m [y^i_{t+1}]_j \langle \nabla g^u_{ij}(x_t^i), x_t^i - x^i  \rangle
\\
\ge {} & 
\left\langle \frac{\partial f^u_0}{\partial x^i}(x_t),   x^i_{t+1} - x^i  \right\rangle  + \langle y_{t+1}^i, \ell _{g^u}^i ({t+1}) - g^u_i(x^i) \rangle .
\label{proof3}
\end{aligned}
\end{equation}
Noting that $\delta _t^{G_i} = V_t^i - v_t^i$ and considering  \eqref{proof1}, \eqref{proof2} and
\eqref{proof3}, we get
\begin{equation}
\begin{aligned}
& \left\langle \frac{\partial f^u_0}{\partial x^i}(x_t), x^i_{t+1} - x^i  \right\rangle  +\langle \delta_t^{G_i}, x_{t+1}^i  - x^i  \rangle
\\
&- \langle g^u_i(x^i), y_{t+1}^i \rangle 
+ \langle \ell_{g^u}^i({t+1}), y_{t+1}^i \rangle  - \langle
s_t^i, y^i_{t+1} - y
\rangle 
\\
\le {} & 
\frac{1}{2\eta_t}\Big[
\|x^i-x^i_t\|^2_2 -  \|x^i_{t+1}- x^i_t\|_2^2 -  \|x^i - x_{t+1}^i\|^2_2
\Big]
\\
&+ \frac{1}{2\mu_t} \Big[ \|y-p_t^i\|_2^2 - \| y_{t+1}^i - p_t^i  \|_2^2 - \| y- y^i_{t+1}  \|^2_2   \Big].
\label{proof4}
\end{aligned}
\end{equation}
Since $g_{ij}^u$ is an $L_{ij}$-smooth function and $L_i = \sqrt{\sum_{j=1}^mL_{ij}^2 }$, we can see that
\begin{equation*}
\begin{aligned}
g^u_i(x^i_{t+1}) - \ell _{g^u}^i({t+1})
&= g^u_i(x^i_{t+1}) - \Big[ g^u_i(x_t) + \nabla g^u_i(x^i_t) (x^i_{t+1} - x^i_t ) \Big]
\le \frac{L_i}{2} \|x^i_{t+1} - x^i_t\|^2,
\end{aligned}
\end{equation*}
and according to the Cauchy--Schwarz inequality, we further get
\begin{align*}
\langle g^u_i(x^i_{t+1}) - \ell_{g^u}^i(t+1), y \rangle \le \|y\|C^i_{t+1},
\end{align*}
where $C^i_{t+1} = \frac{L_i}{ 2} \|x^i_{t+1} - x^i_t\|^2$.
Then
\begin{equation}
\begin{aligned}
\|y\|C^i_{t+1} &= \frac{L_i}{2}( \|y\| - 1 ) \|x^i_{t+1} - x^i_t\|^2 + \frac{L_i}{2}\|x^i_{t+1} - x^i_t\|^2
\\
&\le \frac{L_i}{2} [\|y\|-1]_+\|x^i_{t+1} - x^i_t\|^2 + \frac{L_f}{2}\|x^i_{t+1} - x^i_t\|^2
\\
&\le \frac{L_i}{2}
\|x^i_{t+1} - x^i_t\|^2 + \frac{L_i\bar R_i}{2} [\|y\|-1]_+\|x^i_{t+1} - x^i_t\|.
\label{proof6}
\end{aligned}
\end{equation}
Next, from $q_t^i = \ell _{G}^i(t) - \ell _{G}^i({t-1})$ and $\delta_t^{F_i} = \ell_{G}^i(t) - \ell _{g^u}^i(t)$, we can see that
\begin{equation}
\begin{aligned}
&\langle \ell_{g^u}^i({t+1}), y^i_{t+1} \rangle - \langle g^u_i(x^i_{t+1}), y  \rangle - \langle s^i_t, y^i_{t+1} - y \rangle
\\
= {} & \langle \ell_{g^u}^i({t+1}), y^i_{t+1} \rangle - 
\Big[ \langle g^u_i(x^i_{t+1}) - \ell_{g^u}^i({t+1}), y \rangle   \Big] -
\langle \ell_{g^u}^i({t+1}) , y\rangle -  \langle s^i_t, y^i_{t+1} - y \rangle
\\
\ge {} & 
\langle \ell_{g^u}^i({t+1}), y^i_{t+1} \rangle -
\langle \ell_{g^u}^i({t+1}) , y \rangle - 
\langle s^i_t, y^i_{t+1} - y \rangle - \|y\|C^i_{t+1}
\\
= {} & \langle \ell_{g^u}^i({t+1}) - s^i_t , y^i_{t+1} - y \rangle - \|y\|C^i_{t+1}
\\
= {} & \langle \ell_{g^u}^i({t+1}) - 
\ell_{G}^i(t) - 
\theta_t q^i_t, y^i_{t+1} - y   \rangle - 
\|y\|C^i_{t+1}
\\
= {} & 
\langle \ell_{g^u}^i({t+1}) + 
\ell_{G}^i({t+1})  -
\ell_{G}^i(t) -
\ell_{G}^i({t+1}) - 
\theta_t q^i_t, y^i_{t+1} - 
y
\rangle - \|y\|C^i_{t+1}
\\
= {} & \langle q^i_{t+1}, y^i_{t+1} - y \rangle -  \theta _t\langle  q^i_t, y^i_{t+1} - y \rangle - \langle \delta_{t+1}^{F_i}, y^i_{t+1} - y \rangle - \|y\|C^i_{t+1}
\\
= {} & \langle q^i_{t+1}, y^i_{t+1} - y \rangle -  \theta _t\langle  q^i_t, y^i_{t+1} - p^i_t \rangle - \theta_t\langle q^i_t, p^i_t - y \rangle  - \langle \delta_{t+1}^{F_i}, y^i_{t+1} - y \rangle - \|y\|C^i_{t+1}.
\label{proof5}
\end{aligned}
\end{equation}
Combining \eqref{proof4}, \eqref{proof6} and \eqref{proof5}, we obtain
\begin{equation}
\begin{aligned}
&\left\langle \frac{\partial f^u_0}{\partial x^i}(x_t), x^i_{t+1}-x^i\right\rangle 
+ \langle g^u_i(x^i_{t+1}), y  \rangle  - \langle g^u_i(x^i), y_{t+1}^i \rangle 
+\langle \delta_t^{G_i}, x_{t}^i-x^i\rangle 
- \langle \delta_{t+1}^{F_i}, y^i_{t+1} - y \rangle
\\
&+\langle q^i_{t+1}, y^i_{t+1} - y \rangle  - \theta_t\langle q^i_t, p^i_t - y \rangle  
+\frac{1}{2} L_0\|x_{t+1}^i - x_{t}^i\|^2
\\
\le {} & 
\frac{1}{2\eta_t} \Big[\|x^i- x^i_t\|_2^2-\|x^i-x_{t+1}^i\|^2_2\Big]+
\Big[\frac{1}{2\mu_t}\|y-y_t^i\|_2^2-\frac{1}{2\mu_t}\|y-y^i_{t+1} \|^2_2   \Big] 
+ \theta _t\langle  q^i_t, y^i_{t+1} - p^i_t \rangle
\\
&-\frac{1}{2\mu_t} \| y_{t+1}^i - p_t^i  \|_2^2 
+ \langle \delta_t^{G_i}, x_t^i - x_{t+1}^i \rangle  + \frac{1}{2\mu_t} \Big[ \|y-p_t^i\|_2^2-\|y-y_t^i\|_2^2   \Big]
\\
&+ \mathcal H^i(y) \|x^i_{t+1} - x^i_t\| - \frac{1}{2}\left(\frac{1}{\eta_t} - L_i - L_0 \right)  \| x^i_{t+1}-x^i_t\|_2^2
.
\end{aligned}
\end{equation}
where we denote $\mathcal H^i(y) = \frac{L_i\bar R_i}{2} [\|y\|-1]_+$.
Now since $\bar q_t^i = \ell_{g^u}^i(t) - \ell_{g^u}^i({t-1})$, by
multiplying both sides of the above inequality by $\gamma_t$ and then summing over $t$, we obtain
\begin{equation}
\begin{aligned}
&\sum_{t = 0}^{T-1} \gamma_t \left[ \left\langle\frac{\partial f^u_0}{\partial x^i}(x_t), x^i_{t+1}-x^i\right\rangle 
+ \langle g^u_i(x^i_{t+1}), y  \rangle  - \langle g^u_i(x^i), y_{t+1}^i \rangle +\frac{L_0}{2}\|x_{t+1}^i - x_t^i\|^2_2
\right]
\\
&+
\sum_{t=0}^{T-1}\gamma_t  \Big[ \langle \delta_t^{G_i}, x_{t}^i-x^i\rangle 
- \langle \delta_{t+1}^{F_i}, y^i_{t+1} - y \rangle + \langle q^i_{t+1}, y^i_{t+1} - y \rangle  - \theta_t\langle q^i_t, p^i_t - y \rangle \Big]
\\
\le {} &
\frac{\gamma_0}{2\eta_0} \|x^i - x_0^i\|_2^2 - \frac{\gamma_{T-1}}{2\eta_{T-1}} \| x^i-x_T^i\|_2^2
+ 
\frac{\gamma_0}{2\mu_0}\|y-y_0^i\|_2^2-\frac{\gamma_{T-1}}{2\mu_{T-1}}\|y-y^i_{T} \|^2_2 
\\
&+
\sum_{t= 0}^{T-1} \Big[ \gamma_t\theta_t \langle q_t - \bar q_t^i, y_{t+1}^i - p_{t}^i\rangle + \gamma_t\theta_t\langle \bar q_t^i, y_{t+1}^i - p_t^i \rangle
\Big] -\sum_{t=0}^{T-1} \frac{\gamma_t}{2\mu_t}\|y_{t+1}^i - p_t^i\|^2
\\
&+ \sum_{t=0}^{T-1} \gamma_t \langle \delta_t^{G_i}, x_t^i - x_{t+1}^i \rangle 
+ \sum_{t=0}^{T-1} \frac{\gamma_t}{2\mu_t} \Big[ \|y-p_t^i \|_2^2 - \|y-y_t^i\|_2^2  \Big]
\\
& -  
\sum_{t=0}^{T-1} \left[ \frac{\gamma_t}{2}\left(\frac{1}{\eta_t} - L_0 - L_i\right)\|x_{t+1}^i- x_t^i\|_2^2 - \gamma_t \mathcal H^i(y)\|x_{t+1}^i - x_t^i \| \right].
\end{aligned}
\label{proof7}
\end{equation}
Based on \eqref{gamma_and_theta_relation}, we have
\begin{equation}
\begin{aligned}
&\sum_{t=0}^{T-1} [\gamma_t\langle q^i_{t+1}, y^i_{t+1} - y \rangle  - \gamma_t\theta_t\langle q^i_t, p^i_t - y \rangle  ] 
\\
= {} & 
\sum_{t = 0}^{T-1} \Big[ \gamma_t\langle q_{t+1}^i, y_{t+1}^i - y \rangle -
\gamma_t\theta_t\langle q_t^i, y_t^i-y \rangle -
\gamma_t\theta_t \langle q_t^i, p_t^i - y_t^i \rangle
\Big]
\\
={}& \gamma_{T-1} \langle q_T^i, y_T^i - y \rangle - 
\sum_{t=0}^{T-1} \gamma_t\theta_t\langle q_t^i, p_t^i - y_t^i \rangle.
\end{aligned}
\label{proof8}
\end{equation}
Combining \eqref{proof8} and \eqref{proof7} leads to
\begin{equation}
\begin{aligned}
&\sum_{t = 0}^{T-1} \gamma_t \left[ \left\langle\frac{\partial f^u_0}{\partial x^i}(x_t), x^i_{t+1}-x^i\right\rangle 
+ \langle g^u_i(x^i_{t+1}), y  \rangle  - \langle g^u_i(x^i), y_{t+1}^i \rangle  + \frac{L_0}{2}\|x_{t+1}^i - x_t^i\|^2
\right]
\\
& +
\sum_{t=0}^{T-1}\gamma_t  \Big[ \langle \delta_t^{G_i}, x_{t}^i-x^i\rangle 
- \langle \delta_{t+1}^{F_i}, y^i_{t+1} - y \rangle \Big]
+
\gamma_{T-1} \langle q_T^i, y_T^i - y \rangle 
\\
\le {} &
\frac{\gamma_0}{2\eta_0} \|x^i-x_0^i\|_2^2 - \frac{\gamma_{T-1}}{2\eta_{T-1}} \|x^i-x_T^i\|_2^2
+ 
\frac{\gamma_0}{2\mu_0}\|y-y_0^i\|_2^2-\frac{\gamma_{T-1}}{2\mu_{T-1}}\|y-y^i_{T} \|^2_2 
\\
& +
\sum_{t= 0}^{T-1} \Big[ \gamma_t\theta_t \langle q_t - \bar q_t^i, y_{t+1}^i - p_{t}^i\rangle + \gamma_t\theta_t\langle \bar q_t^i, y_{t+1}^i - p_t^i \rangle
\Big] -  \sum_{t=0}^{T-1} \frac{\gamma_t}{2\mu_t}\|y_{t+1}^i - p_t^i\|^2
\\
& +
\sum_{t=0}^{T-1} \gamma_t \langle \delta_t^{G_i}, x_t^i - x_{t+1}^i \rangle 
+ \sum_{t=0}^{T-1} \frac{\gamma_t}{2\mu_t} \Big[ \|y-p_t^i \|_2^2 - \|y-y_t^i\|_2^2  \Big]
+
\sum_{t=0}^{T-1} \gamma_t\theta_t\langle q_t^i, p_t^i - y_t^i \rangle.
\\
& +
\sum_{t=0}^{T-1} \left[ \frac{\gamma_t}{2}\left(\frac{1}{\eta_t} - L_0 - L_i\right)\|x_{t+1}^i-x_t^i\|_2^2 - \gamma_t \mathcal H^i(y)\|x_{t+1}^i - x_t^i \| \right]
% \\
% & +
% \sum_{t=0}^{T-1} \gamma_t\theta_t\langle q_t^i, p_t^i - y_t^i \rangle
.
\end{aligned}
\label{proof9}
\end{equation}
For $\bar q_t^i$, note that
\begin{equation*}
\begin{aligned}
\|\bar q^i_t\| &= \|\ell _{g^u}^i(t) - \ell_{g^u}^i({t-1})\| 
\\
&= \|  g^u_i(x^i_{t-1}) + \nabla g^u_i(x^i_{t-1})(x^i_t - x^i_{t-1} ) - g^u_i(x^i_{t-2}) - \nabla g^u_i(x^i_{t-2})(x^i_{t-1}-x^i_{t-2})\|
\\
&\le \|g^u_i(x^i_{t-1}) - g^u_i(x^i_{t-2}) \| + \| \nabla g^u_i(x^i_{t-1})^T(x^i_t - x^i_{t-1} )\| + \| \nabla g^u_i(x^i_{t-2})^T(x^i_{t-1}-x^i_{t-2})\|
\\
& \le 2M_i \|x^i_{t-1} - x^i_{t-2}\| + M_i\|x^i_t - x^i_{t-1}\|,
\end{aligned}
\end{equation*}
where we used the fact that $g_{ij}$ is $M_{ij}$-Lipschitz and let
$M_i = \sqrt{ \sum_{j=1}^m M_{ij}^2}$.

Now, based on \eqref{M_gamma_theta_ineq1} and Young's inequality, we have
\begin{equation}
\begin{aligned}
\label{proof10}
&\gamma_t \theta_t \langle \bar q^i_t, y^i_{t+1}-p^i_t \rangle - \frac{\gamma_t }{3\mu_t} \|y^i_{t+1} - p^i_t\|^2 - \frac{\gamma_{t-2 } ( \frac{1}{\eta_{t-2}} - L_0 - L_i)   }{8} \|x^i_{t-1}-x^i_{t-2}\|_2^2 
\\
&- \frac{\gamma_{t-1}( \frac{1}{\eta_{t-1}} - L_0 - L_i)  }{8} \|x^i_t- x^i_{t-1}\|_2^2
\\
\le {} &
 \gamma_t\theta_t \|\bar q^i_t\|_2\|y^i_{t+1}-p^i_t\|_2 - \frac{\gamma_t}{3\mu_t} \|y^i_{t+1}-p^i_t\|^2
-
\frac{\gamma_{t-2}(\frac{1}{\eta_{t-2}}-L_0-L_i)}{8} \|x^i_{t-1}- x^i_{t-2}\|_2^2
\\
&- \frac{\gamma_{t-1}(\frac{1}{\eta_{t-1}}-L_0-L_i)}{8} \|x^i_t-x^i_{t-1}\|_2^2
\\
\le {} &
2M_i \gamma_t\theta_t \|x^i_{t-1}-x^i_{t-2}\| \|y^i_{t+1}-p^i_t\| - \frac{\gamma_t}{6\mu_t}\|y^i_{t+1} - p^i_t \|^2 
\\
&+
M_i\gamma_t \theta_t \|x_t^i - x^i_{t-1}\| \|y^i_{t+1} - p^i_t\| - \frac{\gamma_t}{6\mu_t}\|y^i_{t+1} - p^i_t\|^2 
\\
&
-\frac{\gamma_{t-2} ( \frac{1}{\eta_{t-2}} - L_0 - L_i)}{8}  \| x_{t-1}^i - x_{t-2}^i \|^2
- \frac{\gamma_{t-1}( \frac{1}{\eta_{t-1 }}  L_0- L_i)}{8}  \| x_t^i - x_{t-1}^i \|^2
\\
\le {} &
0.
\end{aligned}
\end{equation}
Applying Cauchy--Schwarz inequality and Young's inequality, we can justify that
\begin{equation}
\begin{aligned}
\gamma_t\theta_t\langle q_t^i - \bar q_t^i, y_{t+1}^i - p_t^i \rangle 
- \frac{\gamma_t }{6\mu _t} \| y_{t+1}^i-p_t^i\|^2 
&\le 
\frac{3\gamma_t\theta_t^2\mu_t}{2} \|q_t^i - \bar q_t^i \|^2
\\
\langle \gamma_t \delta_t^{G_i}, x_{t}^i - x_{t+1}^i \rangle-
\frac{\gamma_t(\frac{1}{\eta_t} - L_0 - L_i)}{8}\|x_{t+1}^i - x_{t}^i\|_2^2
&\le 
\frac{2\gamma_t}{\frac{1}{\eta_t} - L_0 - L_i}\|\delta_t^{G_i}\|^2
\\
\gamma_t\mathcal H^i(y)\|x^i_{t+1} - x^i_t\| -  \frac{\gamma_t(\frac{1}{\eta_t} - L_0 - L_i)}{8}\|x_{t+1}^i- x_t^i\|_2^2
&\le
\frac{2\gamma_t}{\frac{1}{\eta_t} - L_0 - L_i}\mathcal H^i(y)^2.
\end{aligned}
\label{proof11}
\end{equation}
As a result, by combining \eqref{proof9}, \eqref{proof10} and \eqref{proof11}, we can obtain
\begin{equation}
\begin{aligned}
&\sum_{t = 0}^{T-1} \gamma_t \left[ \left\langle\frac{\partial f^u_0}{\partial x^i}(x_t), x^i_{t+1}-x^i\right\rangle 
+ \langle g^u_i(x^i_{t+1}), y  \rangle  - \langle g^u_i(x^i), y_{t+1}^i \rangle + \frac{L_0}{2}\|x_{t+1}^i - x_t^i\|^2
\right]
\\
& + 
\sum_{t=0}^{T-1}\gamma_t  \Big[ \langle \delta_t^{G_i}, x_{t}^i-x^i\rangle 
- \langle \delta_{t+1}^{F_i}, y^i_{t+1} - y \rangle \Big] + \gamma_{T-1} \langle q_T, y^i_T - y \rangle 
\\
\le {} &
\frac{\gamma_0}{2\eta_0} \|x^i-x_0^i\|_2^2 - \frac{\gamma_{T-1}}{2\eta_{T-1}} \|x^i-x_T^i\|_2^2
+ 
\frac{\gamma_0}{2\mu_0} \| y - y_0^i \|^2_2 - 
\frac{\gamma_{T-1}}{2\mu_{T-1}} \| y - y_{T}^i \|_2^2
\\
& + 
\sum_{t=0}^{T-1} \Big[\frac{3\gamma_t\theta_t^2\mu_t}{2} \|q_t^i - \bar q_t^i \|^2
+ \frac{2\gamma_t}{ \frac{1}{\eta_t} - L_0 - L_i}\|\delta_t^{G_i}\|^2 +  \frac{2\gamma_t}{\frac{1}{\eta_t} - L_0 - L_i}\mathcal H^i(y)^2 \Big]
\\
& -
\frac{\gamma_{T-2}(\frac{1}{\eta_{T-2}} - L_0 - L_i)}{8}\|x^i_{T-1}-x^i_{T-2}\|_2^2
- \frac{\gamma_{T-1}(\frac{1}{\eta_{T-1}} - L_0 - L_i)}{4}\|x_{T}^i - x^i_{T-1}\|_2^2
\\
& +
\sum_{t=0}^{T-1} \gamma_t \theta_t \langle q_t^i, p_t^i - y_t^i \rangle + \sum_{t=0}^{T-1}\gamma_t\Big[
\frac{1}{2\mu_t} \| y-p_t^i \|_2^2  - \frac{1}{2\mu_t}\|y-y_t^i\|_2^2
\Big].
\end{aligned}
\label{proof12}
\end{equation}
Similarly, base on \eqref{M_gamma_theta_ineq2}, we get
\begin{equation}
\begin{aligned}
&-\gamma_{T-1} \langle \bar q_T^i, y^i_T - y  \rangle - \frac{\gamma_{T-1}}{3\mu _{T-1}}\|y - y_T^i\|^2
\\
& -
\frac{\gamma_{T-2}(\frac{1}{\eta_{T-2}} - L_0 - L_i)}{8} \|x_{T-1}^i - x_{T-2}^i\|_2^2
-\frac{\gamma_{T-1}(\frac{1}{\eta_{T-1}} - L_0 - L_i)}{4} \|x_{T}^i , x_{T-1}^i\|_2^2
\\
\le {} & M_i\gamma_{T-1} \|x_{T}^i - x_{T-1}^i\| \|y_T^i - y\| - \frac{\gamma_{T-1}}{12\mu_{T-2}}\|y - y_T^i\|^2 
\\
& +  
2M_i\gamma_{T-1}\| x^i_{T-1} - x_{T-2}^i \| \| y_T^i - y\| - \frac{\gamma_{T-1}}{6\mu_{T-1}}\|y - y_T^i\|^2 -\frac{\gamma_{T-1}}{12\mu_{T-1}}\|y_T^i - y\|^2
\\
&
- \frac{\gamma_{T-1}(\frac{1}{\eta_{T-1}} - L_0 - L_i) }{ 4}\|x^i_T- x^i_{T-1}\|_2^2
- \frac{\gamma_{T-2}(\frac{1}{\eta_{T-2}} - L_0 - L_i) }{4}\|x_{T-1}^i-x_{T-2}^i\|^2
\\
\le {} &
-\frac{\gamma_{T-1}}{12\mu_{T-1}}\|y_T^i - y\|^2
,
\end{aligned}
\label{proof13}
\end{equation}
and again using Young’s inequality, we have
\begin{equation}
\begin{aligned}
-\gamma_{T-1} \langle q_T^i - \bar q^i_T, y_T^i - y \rangle - \frac{\gamma_{T-1}}{6\mu_{T-1}} \|y- y_T\|^2 \le \frac{3\gamma_{T-1}\mu _{T-1}}{2} \|q_T^i - \bar q_T^i\|^2.
\end{aligned}
\label{proof14}
\end{equation}

Combining \eqref{proof12}, \eqref{proof13} and \eqref{proof14}, we obtain
\begin{equation}
\begin{aligned}
& \sum_{t = 0}^{T-1} \gamma_t \left[ \left\langle\frac{\partial f^u_0}{\partial x^i}(x_t), x^i_{t+1}-x^i\right\rangle 
+ \langle g^u_i(x^i_{t+1}), y  \rangle  - \langle g^u_i(x^i), y_{t+1}^i \rangle + \frac{L_0}{2}\|x_{t+1}^i - x_t^i\|^2
\right]
\\
&+ 
\sum_{t=0}^{T-1}\gamma_t  \Big[ \langle \delta_t^{G_i}, x_{t}^i-x^i\rangle 
- \langle \delta_{t+1}^{F_i}, y^i_{t+1} - y \rangle \Big]
\\
\le {} & 
\frac{\gamma_0}{2\eta_0} \|x^i-x_0^i\|_2^2 - \frac{\gamma_{T-1}}{2\eta_{T-1}} \|x^i - x_T^i\|_2^2
+ 
\frac{\gamma_0}{2\mu_0} \| y - y_0^i \|^2_2 - 
\frac{\gamma_{T-1}}{12\mu_{T-1}} \| y - y_{T}^i \|_2^2
\\
 & +
\sum_{t=0}^{T-1} \frac{2\gamma_t}{ \frac{1}{\eta_t} - L_0 - L_i} \Big[\|\delta_t^{G_i}\|^2
 +  \big( \frac{L_i\bar R_i}{2} \big[\|y\|-1 \big]_+   \big)^2 \Big]  
+
\sum_{t=0}^{T-1} \frac{3\gamma_t\theta_t^2\mu_t}{2} \|q_t^i - \bar q_t^i \|^2 
\\
&
+
\frac{3\gamma_{T-1}\mu _{T-1}}{2} \|q_T^i - \bar q_T^i\|^2
%\\
%& 
+
\sum_{t=0}^{T-1} \gamma_t \theta_t \langle q_t^i, p_t^i - y_t^i \rangle + \sum_{t=0}^{T-1}\gamma_t\Big[
\frac{1}{2\mu_t} \| y-p_t^i \|_2^2  - \frac{1}{2\mu_t}\|y-y_t^i\|_2^2
\Big].
\end{aligned}
\label{proof15}
\end{equation}

Next, we note that
\begin{equation}
\begin{aligned}
\sum_{i=1}^n \langle   \frac{\partial f^u_0}{\partial x^i}(x_t) , x_{t+1}^i - x^i \rangle
&= \langle \nabla f^u_0(x_t) , x_{t+1} - x\rangle 
\\
&= \langle \nabla f^u_0(x_t), x_{t+1} - x_t\rangle  +  \langle \nabla f^u_0(x_t), x_{t}  - x\rangle
\\
& \ge  \Big[ f^u_0(x_{t+1}) - f^u_0(x_t) - \frac{L_0}{2}\|x_{t+1} - x_t\|^2  \Big]
+ \Big[ f^u_0(x_t) - f^u_0(x)  \Big]
\\
&
= f^u_0(x_{t+1}) - f^u_0(x) - \frac{L_0}{2}\|x_{t+1} - x_t\|^2,
\end{aligned}
\label{proof16}
\end{equation}
where we used the smoothness and convexity of $f_0$. Then, due to convexity of $f(x) = \|x\|^2$ and using Jensen's inequality, we have
\begin{align*}
\sum_{i=1}^n \| y - p_t^i \|_2^2 \le \sum_{i=1}^n \sum_{j=1}^n W_{ij} \| y-y_t^j \|_2^2 = \sum_{j=1}^n \|y-y_t^j \|_2^2.
\end{align*}
The last equality holds because we exchanged the order of summation.
Thus,
\begin{equation}
\begin{aligned}
&\sum_{i=1}^n \sum_{t=0}^{T-1} \frac{\gamma_t }{2\mu_t} \| y - p_t^i \|_2^2- \sum_{i=1}^n\sum_{t=0}^{T-1} \frac{\gamma_t }{2\mu_t}\| y - y_t^i \|_2^2 
\le 0.
\end{aligned}
\label{proof17}
\end{equation}

Finally, by combining \eqref{proof15} with \eqref{proof16}, summing over $i$ and plugging in the inequality \eqref{proof17}, we complete the proof.
\end{proof}

\subsection{Building Blocks for Convergence Rate and Constraint Violation Analysis}
In this subsection, we will derive the initial forms of the convergence rate and constraint violation bounds using Lemma~\ref{lemma:2}.
\begin{lemma}
\label{lemma:5}
For all $T\ge 1$, we have
\begin{equation}
\begin{aligned}
\label{initial_conver_analy}
&\Gamma_T \mathbb E  \Big[f_0(\bar x_T) -  f_0(x^*) \Big] 
\\
\le {} &
\frac{\gamma_0}{2\eta_0}  \|x^{*, u}- x_0\|_2^2  + 
\sum_{i=1}^n 
\frac{\gamma_0}{2\mu_0}\|y_0^i\|^2 
+
\sum_{i=1}^n \Bigg[
\sum_{t= 0}^{T-1} \frac{2\gamma_t}{\frac{1}{\eta_t} - L_0 - L_i}  
\mathbb E
[\|\delta_t^{G_i}\|^2]
\\
&+
\frac{3\gamma_{T-1}\mu _{T-1}}{2} \mathbb E\big[ \|q_T^i - \bar q_T^i\|^2 \big]
+
\sum_{t=1}^{T-1}  \frac{3\gamma_t\theta_t^2\mu_t}{2} \mathbb E\big[ \|q_t^i-\bar q_t^i\|^2
\big]
\Bigg]
\\
&-
\sum_{t=0}^{T-1}\gamma_t  \mathbb E \Big[ \sum_{i=1}^n \langle \delta_t^{G_i}, x^i_{t}-x^i  \rangle-
\sum_{i=1}^n
\langle \delta_{t+1}^{F_i}, y^i_{t+1} - y \rangle \Big]_{x = x^{*, u}, y = 0}
\\
&+
\mathbb E\Big[ \sum_{i=1}^n\sum_{t=0}^{T-1} \gamma_t\theta_t\langle q_t^i, p_t^i - y_t^i \rangle \Big]
+ \Gamma_T \frac{1}{2} u^2 L_0d, 
\end{aligned} 
\end{equation}
where $\Gamma_T = \sum_{t= 0}^{T-1}\gamma_t$.
\end{lemma}

\begin{proof}
By Jensen's inequality, we have
\begin{equation*}
 \begin{aligned}
f_0(\bar x_T) &= f_0\!\left( \frac{ \sum_{t=0}^{T-1}  \gamma_t x_{t+1}     } {\sum_{t = 0}^{T-1} \gamma_t}  \right) 
\le \frac{1 }{\sum_{t=0}^{T-1}\gamma_t}\sum_{t = 0}^{T-1}\Big[ \gamma_t f_0(x_{t+1}) \Big]
= \frac{1}{\Gamma_T} \sum_{t = 0}^{T-1}\Big[ \gamma_t f_0(x_{t+1})\Big].
\end{aligned}
\label{proof21}
\end{equation*}
Taking $x = x^{*, u}$ and $y = 0$, we see that the first line on the left-hand side of~\eqref{proof18} becomes
\begin{equation*}
\begin{aligned}
&\sum_{t = 0}^{T-1} \gamma_t \Big[ 
f^u_0(x_{t+1}) - f^u_0(x) + 
\sum_{i=1}^n\langle g^u_i(x^i_{t+1}), y\rangle  - 
\sum_{i=1}^n\langle g^u_i(x^i), y_{t+1}^i\rangle 
 \Big]_{| x = x^{*, u}, y = 0}
\\
= {} &  \sum_{t = 0}^{T-1} \gamma_t \Big [ f^u_0(x_{t+1}) - f^u_0(x^{*, u}) - \sum_{i=1}^n \langle g_i^u( (x^i)^{*, u} ), y^i_{t+1} \rangle  \Big].
\end{aligned}
\label{proof22}
\end{equation*}
Then, from \eqref{proof19} and \eqref{proof20}, we have
\begin{equation*}
\begin{aligned}
| f_0^u(x^{*, u}) - f_0(x^*) | \le \frac{1}{2}  u^2 L_0 d.
\end{aligned}
\label{proof23}
\end{equation*}
Summarizing the above inequalities, we obtain
\begin{equation}
\begin{aligned}
\Gamma_T \Big[f_0(\bar x_T) - \big( f_0(x^*) + \frac{1}{2}u^2L_0d  \big) \Big] &\le \sum_{t = 0}^{T-1} \gamma_t \Big [ f^u_0(x_{t+1}) - f^u_0(x^{*, u})  \Big]
\\
&\le  \sum_{t = 0}^{T-1} \gamma_t \Big [ f^u_0(x_{t+1}) - f^u_0(x^{*, u}) - \sum_{i=1}^n \langle g_i^u( (x^i)^{*, u} ), y^i_{t+1} \rangle  \Big],
\end{aligned}
\label{proof24}
\end{equation}
where the second inequality follows by noting that
$g_i^u ((x^i)^{*, u} ) \le  0$
and $y_{t+1}^i\ge  0$ leads to the inequality 
$\sum_{i=1}^n \langle g_i^u( (x^i)^{*, u} ), y^i_{t+1} \rangle \le 0$.

By combining \eqref{proof18} and (\ref{proof24}), we complete the proof.
\end{proof}

\begin{lemma}
[\cite{boob2023stochastic}]
\label{lemma:6}
% \cite{boob2023stochastic}
Let $\rho_0, \dots, \rho_j$ be a sequence in $\mathbb{R}^n$, and let $S$ be a convex set in $\mathbb{R}^n$. Define the sequence $v_t$ for $t = 0, 1, \dots$ such that $v_0 \in S$ and
\begin{align}
v_{t+1} = \arg \min _{x\in S} \langle \rho_t, x \rangle + \frac{1}{2} \|x-v_t \|_2^2.
\end{align}
Then, for any $x \in S$ and $t \geq 0$, we have
\begin{align}
\langle \rho_t, v_t-x \rangle &\le \frac{1}{2} \|x-v_t \|_2^2 - 
\frac{1}{2}\| x - v_{t+1} \|_2^2 + \frac{1}{2}\|\rho_t \|_2^2
\\
\sum_{t=0}^j \langle \rho_t, v_t - x \rangle &\le 
\frac{1}{2}\|x-v_0 \|_2^2 + \frac{1}{2}\sum_{t=0}^j \|\rho_t \|_2^2 .
\end{align}
\end{lemma}

Before proceeding to the next lemma that provides a preliminary form for the constraint violation bound, we introduce some extra notations. We let $R\coloneqq \|y^{*, u} \|+1$ and
\[
\mathcal{B}_+^2(R) = \{ x \in \mathbb{R}^m : \|x\|_2 \leq R, \, x \geq 0 \}.
\]
We also denote
\[
\hat y \coloneqq (\|y^{*, u}\| + {1}) \frac{[\sum_{i=1}^n g^u_i(\bar x^i_T)]_+}{\|[\sum_{i=1}^ng^u_i(\bar x^i_T)]_+ \|} \in \mathcal B_+^2(R).
\]
We further define an auxiliary sequence $(y^i_v(t))_{t\geq 0}$ for each $i$ such that $y^i_v(0) \coloneqq y^i_0$, and for all $t\ge 0$,
\begin{align*}
y^{i}_v(t+1) \coloneqq \argmin_{y\in \mathcal B_+^2(R)} \Big\{ \mu_{t-1} 
\langle\delta_t^{F_i}, y \rangle + \frac{1}{2} \| y - y^i_v (t) \| \Big\}.
\end{align*}
\begin{lemma}
\label{lemma:7}
For all $T\ge 1$, we have
\begin{equation}
\begin{aligned}
\label{initial_viola_analy}
\Gamma_T \mathbb E \left[ \left\|  \left[
\sum_{i=1}^n g_i(\bar x^i_T) 
\right]_+ \right\| \right]
&\le 
\frac{\gamma_0}{2\eta_0} \mathcal \|x^{*, u}- x_0\|_2^2
+ 
\sum_{i=1}^n \frac{\gamma_0}{2\mu_0}  \Big[
\mathbb E\big[\| \hat y-y_0^i\|^2 + \|\hat y - y^i_v(1) \|^2 \big]  \Big]
\\
&
+ 
\sum_{i=1}^n \Bigg[
\sum_{t= 0}^{T-1} \frac{2\gamma_t}{\frac{1}{\eta_t} - L_0 - L_i} \Big[ 
\mathbb E[\|\delta_t^{G_i}\|^2 ]+
\big( \frac{L_i\bar R_i}{2}\| y^{*, u}\|_2     \big)
\Big] 
\\
&
+\frac{3\gamma_{T-1}\mu _{T-1}}{2} \mathbb E\big[  \|q_T^i - \bar q_T^i\|^2\big]
+
\sum_{t=1}^{T-1}  \frac{3\gamma_t\theta_t^2\mu_t}{2} \mathbb E\big[\|q_t^i-\bar q_t^i\|^2
\big]\Bigg]
\\
&
-
\sum_{t=0}^{T-1}\gamma_t \mathbb E \Big[ \sum_{i=1}^n \langle \delta_t^{G_i}, x^i_{t}-x^i \rangle-
\sum_{i=1}^n
\langle \delta_{t+1}^{F_i}, y^i_{t+1} - y _v^i(t+1)\rangle \Big]_{x = x^{*, u}, y = \hat y}
\\
&
+\mathbb E\Big[\sum_{i=1}^n\sum_{t=0}^{T-1} \gamma_t\theta_t\langle q_t^i, p_t^i - y_t^i \rangle \Big]
+\sum_{i=1}^n\mathbb E \Big[ \sum_{t=0}^{T-1}
\frac{\gamma_t\mu_t}{2}\|\delta_{t+1}^{F_i} \|_2^2
\Big].
\end{aligned}
\end{equation}
\end{lemma}
\begin{proof}
According to Lemma \ref{lemma:6}, we have
\begin{equation*}
\begin{aligned}
\mu_t \langle \delta_{t+1}^{F_i}, y_v^i(t+1) - y \rangle 
\le \frac{1}{2} \|y-y^i_v(t+1) \|_2^2 - \frac{1}{2}\|y - y^i_v(t+2) \|_2^2
+ \frac{\mu _t^2}{2} \|\delta_{t+1}^{F_i} \|_2^2.
\end{aligned}
\end{equation*}
Multiplying both sides of the inequality by 
$\frac{\gamma_t}{\mu_t}$ and summing from 
$t = 0$ to $T-1$, we obtain
\begin{equation*}
\sum_{t=0}^{T-1}\!\frac{\gamma_t\mu_t}{\mu_t} 
\langle\delta_{t+1}^{F_i}, y^i_v(t\!+\!1) \!-\! y\rangle 
 \le
\sum_{t=0}^{T-1}\! \frac{\gamma_t}{2\mu_t}\| y \!-\! y_v^i(t\!+\!1) \|_2^2 
-
\sum_{t=0}^{T-1}\! \frac{\gamma_t}{2\mu_t} \| y \!-\! y_v^i(t\!+\!2) \|_2^2 
+
\sum_{t=0}^{T-1}\frac{\gamma_t\mu_t^2}{2\mu_t} \| \delta_{t+1}^{F_i} \|_2^2.
\end{equation*}
Summing over $i$ and noting the second relation in \eqref{gamma_and_theta_relation}, we obtain
\begin{align}
\sum_{i=1}^n\sum_{t=0}^{T-1}\gamma_t \langle\delta_{t+1}^{F_i}, y^i_v(t+1)-y \rangle
\le \sum_{i=1}^n \frac{\gamma_0}{2\mu_0}\|y-y^i_v(1) \|_2^2 + 
\sum_{i=1}^n\sum_{t=0}^{T-1} \frac{\gamma_t\mu_t}{2}\| \delta_{t+1}^{F_i} \|_2^2.
\label{proof25}
\end{align}
From \eqref{proof25} and \eqref{proof18}, we further get
\begin{equation}
\begin{aligned}
&
\sum_{t = 0}^{T-1} \gamma_t \Big[ 
f^u_0(x_{t+1}) - f^u_0(x) + 
\sum_{i=1}^n\langle g^u_i(x^i_{t+1}), y\rangle  - 
\sum_{i=1}^n\langle g^u_i(x^i), y_{t+1}^i\rangle  \Big]
\\
&+ 
\sum_{t=0}^{T-1}\gamma_t  \Big[ \sum_{i=1}^n \langle \delta_t^{G_i}, x^i_{t}-x^i\rangle-
\sum_{i=1}^n
\langle \delta_{t+1}^{F_i}, y^i_{t+1} - y^i_v(t+1) \rangle \Big]
\\
\le {} &  
\frac{\gamma_0}{2\eta_0} \|x, x_0\|_2^2 
+ 
\sum_{i=1}^n \frac{\gamma_0}{2\mu_0} \Big[
\|y-y_0^i\|^2 + \|y - y^i_v(1) \|^2  \Big]
\\
&+
\sum_{i=1}^n \Bigg[
\sum_{t= 0}^{T-1} \frac{2\gamma_t}{ \frac{1}{\eta_t} - L_0 - L_i} \Big[ 
||\delta_t^{G_i}||^2+
\big( \frac{L_i\bar R_i}{2}[\|y\|-1]_+   \big)^2
\Big] 
+
\frac{3\gamma_{T-1}\mu_{T-1}}{2} \|q_T^i - \bar q_T^i\|^2 \\
& +
\sum_{t=1}^{T-1}  \frac{3\gamma_t\theta_t^2\mu_t}{2}\|q_t^i-\bar q_t^i\|^2
+ \sum_{t=0}^{T-1}\frac{\gamma_t\mu_t}{2}\| \delta_{t+1}^{F_i} \|_2^2
\Bigg]
+
\sum_{i=1}^n\sum_{t=0}^{T-1} \gamma_t\theta_t\langle q_t^i, p_t^i - y_t^i \rangle.
\end{aligned}
\label{proof26}
\end{equation}
Then, by Jensen's inequality, we have
\begin{equation}
\begin{aligned}
&\mathbb E \Big[f^u_0(\bar x_T)  - 
f^u_0(x^{*,u}) + 
 \sum_{i=1}^n \langle g^u_i( \bar x_T ^i ) , \hat y \rangle
-\sum_{i=1}^n\langle  g_i^u( (x^i)^{*, u}), \bar y_T  \rangle
\Big] 
\\
\le {} &
\frac{1}{\Gamma_T} \mathbb E \Bigg[ 
\sum_{t = 0}^{T-1} \gamma_t \Big[ 
f^u_0(x_{t+1})- f^u_0(x^{*,u})  +  \sum_{i=1}^n\langle g^u_i(x^i_{t+1}), 
\hat y \rangle - \sum_{i=1}^n \langle g_i^u( (x^i)^{*, u} ), y_{t+1}^i \rangle  \Big]\Bigg]
\end{aligned}
\label{proof32}
\end{equation}
Note that, for the problem \eqref{problem:2}, its Lagrangian function is given by
\begin{align*}
\mathcal L^u(x, y)  =  f^u_0(x) + \left\langle \sum_{i=1}^n g^u_i(x^i), y \right\rangle.
\end{align*}
Since $
\mathcal L^u(\bar x_T, y^{*, u}) - \mathcal L^u(x^{*, u}, y^{*, u}) \ge 0$, we can see that
\begin{align}
f^u_0(\bar x_T) + \left\langle \sum_{i=1}^n g^u_i(\bar x^i_T), y^{*, u} \right\rangle - f^u_0(x^{*, u}) \ge 0 .
\label{proof27}
\end{align}
Additionally, we have
\begin{align}
\left\langle y^{*, u} , \sum_{i=1}^n g^u_i(\bar x^i_T) \right\rangle \le 
\left\langle y^{*, u}, \left[\sum_{i=1}^n g^u_i(\bar x_T^i)\right]_+ \right\rangle \le 
\|y^{*, u}\|_2    \left\| \left[ \sum_{i=1}^n g^u_i(\bar x^i_T) \right]_+  \right\|_2.
\label{proof28}
\end{align}
Combining \eqref{proof27}) and \eqref{proof28}, 
we obtain
\begin{align*}
f^u_0(\bar x_T) + \|y^{*, u}\|_2   \left\|  \left[\sum_{i=1}^n g^u_i(\bar x^i_T)\right]_+   \right\|_2 - f^u_0(x^{*, u}) \ge 0 .
\label{proof29}
\end{align*}
Moreover,
\begin{equation*}
\begin{aligned}
&\mathcal L^u( \bar x_T, \hat y  ) - \mathcal L^u( x^{*, u}, \bar y_T )  
\\
\ge {} & 
\mathcal L^u ( \bar x_T, \hat y ) - \mathcal L^u( x^{*, u}, y^{*,u} )
\\
= {} & f^u_0(\bar x_T) + (\|y^{*,u}\|_2+1) 
\left\|  \left[
\sum_{i=1}^n g_i(\bar x^i_T) 
\right]_+ \right\|
-f_0^u(x^{*, u})
\\
= {} & 
f^u_0(\bar x_T) + \|y^{*, u}\|_2 
\left\|  \left[
\sum_{i=1}^n g^u_i(\bar x^i_T) 
\right]_+ \right\|
-f^u_0(x^{*, u})
+ \left\|  \left[
\sum_{i=1}^n g^u_i(\bar x^i_T) 
\right]_+ \right\|.
\label{proof30}
\end{aligned}
\end{equation*}
Therefore,
\begin{equation}
\begin{aligned}
\left\|  \left[
\sum_{i=1}^n g^u_i(\bar x^i_T) 
\right]_+ \right\|
 &\le \mathcal L^u( \bar x_T, \hat y  ) - \mathcal L^u( x^{*, u}, \bar y_T ) 
 \\
&=f^u_0(\bar x^i_T) + \left\langle \sum_{i=1}^n g^u_i( \bar x_T ^i ) , \hat y \right\rangle - f^u_0(x^{*,u}) -
\left\langle \bar y_T, \sum_{i=1}^n g_i^u(x^i)^{*, u} \right\rangle.
\end{aligned}
\label{proof31}
\end{equation}
Combining \eqref{proof32} and \eqref{proof31}, 
we obtain
\begin{equation}
\begin{aligned}
&
\mathbb E\!\left[ \left\|  \left[
\sum_{i=1}^n g^u_i(\bar x^i_T) 
\right]_+ \right\| \right] 
\\
\le {} & 
\frac{1}{\Gamma_T} \mathbb E \Bigg[ 
\sum_{t = 0}^{T-1} \gamma_t \Big[ 
f^u_0(x_{t+1})- f^u_0(x^{*,u})  +  \sum_{i=1}^n\langle g^u_i(x^i_{t+1}), 
\hat y \rangle - \sum_{i=1}^n \langle g_i^u( (x^i)^{*, u} ), y_{t+1}^i \rangle  \Big]\Bigg]
.
\label{proof33}
\end{aligned}
\end{equation}
Additionally, since
\begin{align}
\left\|\left[\sum^n_{i=1} g_i(\bar x^i_T)\right]_+ \right\| \le  \left\| \left[\sum_{i=1}^n g_i^u(\bar x^i_T)\right]_+ \right\|,
\label{proof34}
\end{align}
by \eqref{proof33}, \eqref{proof34} and \eqref{proof26}, and setting $x = x^{*, u}$ and
$y = \hat y$, we complete the proof.
\end{proof}
\subsection{Bounds for Certain Quantities in Convergence Rate and Constraint Violation Analysis}
\begin{lemma}
[\cite{tang2023zeroth}]
For any $t\ge 0$, we have
\begin{equation}
\begin{aligned}
\mathbb E[\| D_j(t)z^i_t\|] \le 12 M_0^2 d_i,
\quad \mathbb E[\|G^i_0(t)\|]\le 12 M_0^2 d_i.
\end{aligned}
\end{equation}
\end{lemma}
%This lemma primarily serves to bound the error arising from communication delays.
\begin{lemma}
Suppose $u\le \frac{M_g}{(d+6)L_g }$. For any $t \ge 0$, we have
\begin{align}
\mathbb E[\|V^i_t \|^2] \le 24M_0^2 d_i + 27 
 M_g^2 d_i C^2.
\end{align}
\end{lemma}
\begin{proof}
Given that $V^i_t = G^i_0(t) + \sum_{j=1}^m H_{ij} [y^i_{t+1}]_j$, according to \cite{nesterov2017random}, we have
\begin{equation*}
\begin{aligned}
\mathbb E[\|V^i(t) \|^2] &\le 2\mathbb E[\|G^i(t)\|^2] +
2 \sum_{j=1}^m\mathbb E[\|H_{ij}(t)[y_{t+1}^i]_j \|^2]
\\
&\le 24M_0^2 d_i + \sum_{j=1}^m  C^2 [ u^2 L_{ij}^2(d_i+6)^3 + 4(d_i+4)M_{ij}^2   ]
\\
&= 24M_0^2 d_i + C^2 [u^2 L_i^2(d_i+6)^3 + 4(d_i+4)M_i^2 ].
\end{aligned}
\end{equation*}
Then, since $u\le \frac{M_g}{(d+6)L_g}$, we conclude that
\[
\mathbb E[\|V^i_t \|^2] \le 24M_0^2 d_i + C^2[(d_i+6)M_g^2 + 4(d_i+4)M_i^2]
\le 
24M_0^2 d_i + 27 M_g^2 d_i C^2.
\qedhere
\]
\end{proof}
\begin{lemma}
\label{lemma:9}
Suppose $u\le \frac{M_g}{(d+6)L_g}$ and
let $\frac{1}{\eta_t} = L_0 + L_{\max} + \frac{1}{\eta}$. For any $t\ge 0$, we have
\begin{align}
\mathbb E[\|x^i_t - x^i_{\tau_j^i(t)} \|^2] &\le 
\eta ^2  b_{ij}^2\Big[ 24M_0^2 d_i + 27 M_g^2 d_i C^2 \Big]
\\
\mathbb E[\|x_t - x_{\tau_j^i(t)} \|^2] &\le \eta ^2  b_{ij}^2\Big[ 24M_0^2 d + 
27 M_g^2 d C^2\Big]
.
\end{align}
\end{lemma}

\begin{proof}
First of all, since $\frac{1}{\eta_t} = L_0 + L_{\max} + \frac{1}{\eta}$, we see that $\eta_t = \frac{1}{L_0 + L_{\max} + \frac{1}{\eta}} = \frac{\eta}{ L_0 \eta + L_{\max}\eta + 1} \le \eta$.

The first-order optimality condition for (\ref{eq:primal_update}) is
\begin{align*}
\left\langle V_t^i+\frac{1}{\eta_t}(x_{t+1}^i - x_t^i), x^i - x_{t+1}^i  \right\rangle  \ge  0
\end{align*}
for any $x^i\in\mathcal{X}_i$. Setting $x^i = x_t^i$, we obtain
\begin{equation*}
\begin{aligned}
\langle \eta_t V^i_t , x^i_t - x_{t+1}^i \rangle 
\ge \|x_{t+1}^i-x_t^i \|^2,
\end{aligned}
\end{equation*}
which implies $\|x_{t+1}^i - x_t^i \| \le \eta_t \|V^i_t \|$.

Then we have
\begin{equation*}
\begin{aligned}
\mathbb E[\|x^i_t - x^i_{\tau_j^i(t)} \|^2] 
&\le \mathbb E\Bigg[ 
\Big(\sum_{\tau = -b_{ij}}^{-1} \|\eta_{\tau+t}  V^i(\tau+t) \| \Big)^2
 \Bigg] 
\\
&\le 
\sum_{\tau = -b_{ij}}^{-1} \eta_{\tau+t}^2 \mathbb E[\|V^i(\tau+t) \|^2]
\\
&\le b_{ij}\Big[ 24M_0^2 d_i + 27 
 M_g^2 d_i C^2 \Big] \sum_{\tau = -b_{ij}}^{-1} \eta_{\tau+t} ^2
\\
& \le \eta ^2  b_{ij}^2\Big[ 24M_0^2 d_i + 27 
 M_g^2 d_i C^2  \Big],
\end{aligned}
\end{equation*}
which further leads to
\begin{align*}
\mathbb E[\|x_t - x_{\tau_j^i(t)} \|^2] \le \eta ^2  b_{ij}^2\Big[ 24M_0^2 d + 
27 M_g^2 d C^2
\Big].
\end{align*}
This completes the proof.
\end{proof}

\begin{lemma}
Suppose $u\le \frac{M_g}{(d+6)L_g}$ and
let $\frac{1}{\eta_t} = L_0 + L_{\max} + \frac{1}{\eta}$. For any $t\ge 0$, we have
\begin{equation}
\begin{aligned}
\mathbb E\!\left[\sum_{i=1}^n\langle \delta_t^{G_i} , x^i - x^i_t \rangle\right]
\le 
\eta( M_0 \bar {\mathfrak b} \sqrt d + 
 L_0 \bar b  d \bar R + 2\sqrt 3 \bar {\mathfrak b} d M_0)
(24M_0^2 + 27M_g^2 C^2)^{\frac 1 2}
.
\label{proof36}
\end{aligned}
\end{equation}
and
\begin{equation}
\begin{aligned}
\sum_{i=1}^n \mathbb E[\|\delta_t^{G_i}\|^2]
\le 
48 M_0^2 d + 4 M_0^2n + 62M_g^2d C^2
\label{proof37}
\end{aligned}
\end{equation}
\end{lemma}
\begin{proof}

Recalling the definition of $\delta_t^{G_i}$, we have
\begin{equation}
\begin{aligned}
\label{proof39}
\sum_{i=1}^n\langle \delta_t^{G_i} , x^i - x^i_t \rangle ={}& 
\sum_{i=1}^n\left\langle (G_0^i(t) - \frac{\partial f^u_0}{\partial x^i}(x_t)) + 
(H_i(t)-\nabla g^u_i(x_t^i))^T{y_{t+1}^i}, x^i - x_t^i \right\rangle  
\\
={}& \sum_{i=1}^n\left\langle (G_0^i(t) -  \frac{\partial f^u_0}{\partial x^i}(x_t)),  x^i-x_t^i \right\rangle  
+\sum_{i=1}^n \langle (H_i(t)-\nabla g^u_i(x_t^i))^T{y_{t+1}^i}, x^i - x_t^i \rangle .
\end{aligned}
\end{equation}
Let's consider these two terms separately.
For the first term, we have
\begin{equation}
\begin{aligned}
\label{proof40}
\mathbb E  \left[\sum_{i=1}^n\left\langle G_0^i(t) - \frac{\partial f^{u}_0}{\partial x^i}(x_t) ,
x^i - x_t^i
\right\rangle \right] 
&= 
\mathbb E\!\left[\sum_{i=1}^n \langle G_0^i(t), x^i - x^i_t \rangle\right] -
\mathbb E\!\left[\sum_{i=1}^n \left\langle \frac{\partial f^{u}_0}{\partial x^i}(x_t), x^i - x_t^i   \right\rangle\right].
\end{aligned}
\end{equation}
Recall that $G_0^i(t) = \frac 1n \sum_{j=1}^n D^i_j(t)z^i_{\tau^i_j(t)}$, and we get
\begin{equation}
\begin{aligned}
\label{proof38}
\mathbb E \!\left[\sum_{i=1}^n\langle G_0^i(t) , x^i - x_t^i  \rangle \right]  ={} 
&\mathbb E\!\left[ \frac{1}{n}\sum_{i,j=1}^n \langle D_j(\tau^i_j(t))z^i_{\tau_j^i(t)}, x^i - x^i_{\tau^i_j(t)} \rangle  \right]
\\
&+\mathbb E\!\left[ \frac{1}{n}\sum_{i,j =1}^n \langle  D_j(\tau_j^i(t))z^i_{\tau^i_j(t)}, x^i_{\tau^i_j(t)} -x^i_t \rangle \right].
\end{aligned}
\end{equation}
Then we note that
\begin{equation}
\begin{aligned}
\label{proof44}
& \mathbb E\!\left[ \frac{1}{n}\sum_{i,j=1}^n \langle D_j(\tau^i_j(t))z^i_{\tau_j^i(t)}, x^i - x^i_{\tau^i_j(t)} \rangle  \right] \\
= {} &
\frac 1n\sum_{i,j=1}^n \mathbb E[\langle \nabla^i f_j^u (x_{\tau^i_j(t)}), x^i - x^i_{\tau_j^i(t)} \rangle ] 
\\
= {} &
\mathbb E[ \langle \nabla f^u(x_t), x - x_t \rangle ] 
+
\frac{1}{n} \sum_{i,j=1}^n \mathbb E [\langle \nabla^i f_j^u(x_t), x^i_t - x^i_{\tau^i_j(t)}  \rangle]
\\
& +
\frac{1}{n}\sum_{i,j=1}^n \mathbb E[\langle \nabla ^if_j^u(x_{\tau_j^i(t))} - \nabla ^if_j^u(x_t),  x^i - x^i_{\tau_j^i(t)} \rangle]
.
\end{aligned}
\end{equation}
By the Peter--Paul inequality and Lemma \ref{lemma:9}, we see that
\begin{equation*}
\begin{aligned}
& \frac{1}{n} \sum_{i,j=1}^n \mathbb E [\langle \nabla^i f_j^u(x_t), x^i_t - x^i(\tau^i_j(t))  \rangle]
\\
\le{}&
\frac{1}{n} \sum_{i,j = 1}^n 
\Big[ \frac{\varepsilon}{2} \mathbb E [ \|\nabla ^if_j^u(x_t) \|^2 ] 
+ \frac{1}{2\varepsilon} \mathbb E[\|x^i(t) - x^i_{\tau_j^i(t)} \|^2]   \Big]
\\
\le {} &
\frac{1}{n}
\Big[
\frac{\varepsilon}{2} n M_0^2 + 
\frac{1}{2\varepsilon} \sum_{i,j=1}^n
\eta^2b_{ij}^2d_i(24M_0^2+27M_g^2C^2)
\Big]
\\
= {} &
\frac{\varepsilon}{2}M_0^2 + 
\frac{1}{2\varepsilon}\eta^2 \bar {\mathfrak b}^2 d (24M_0^2+27M_g^2C^2)
.
\end{aligned}
\end{equation*}
Taking $\varepsilon = 
\frac{\eta\bar {\mathfrak b} \sqrt d(24M_0^2 + 27M_g^2 C^2)^{\frac{1}{2}}}{M_0}
$, we get
\begin{equation}
\begin{aligned}
\label{proof41}
\frac{1}{n} \sum_{i,j=1}^n \mathbb E [\langle \nabla^i f_j^u(x_t), x^i_{t} - x^i_{\tau^i_j(t)}  \rangle]
\le
\eta M_0 \bar {\mathfrak b} \sqrt d(24M_0^2 + 27M_g^2 C^2)^{\frac{1}{2}}
.
\end{aligned}
\end{equation}
Next, we note that
\begin{equation}
\begin{aligned}
\label{proof42}
&\frac{1}{n}\sum_{i,j=1}^n \mathbb E[\langle \nabla ^if_j^u(x_{\tau_j^i(t))} - \nabla ^if_j^u(x_t),  x^i - x^i_{\tau_j^i(t)} \rangle] 
\\
\le {} &
\frac{1}{ n} \sum_{i,j = 1}^n 
\mathbb E[\| \nabla^i f_j^u(x_{\tau_j^i(t))} - \nabla^if_j^u(x_t) \| \bar R_i]
\\
\le {} &
\frac{L_0 }{n} \sum_{i,j=1}^n \sqrt{
\mathbb E[\|x_{\tau_j^i(t)} - x_t \|^2] }\bar R_i
\\
\le {} &
\frac{L_0}{n} \sum_{i, j = 1}^n \sqrt{\eta ^2  b_{ij}^2 d\Big[ 24M_0^2 + 27 M_g^2C^2  \Big]}  \bar R_i
\\
= {} &
\eta L_0 \bar b  \sqrt{nd} (24M_0^2 + 27M_g^2 C^2) \bar R
.
\end{aligned}
\end{equation}
Furthermore, by Peter--Paul inequality, we have
\begin{equation*}
\begin{aligned}
&\mathbb E\!\left[\frac{1}{n}\sum_{i,j=1}^n \langle D_j(\tau^i_j(t))z^i_{\tau^i_j(t)}, x^i_{\tau^i_j(t)}-x^i_t \rangle\right] 
\\ 
\le {} &  
\frac{1}{n}\sum_{i,j=1}^n \mathbb E\Big[
\frac{\varepsilon}{2} \|D_j(\tau_j^i(t))z^i_{\tau_j^i(t)} \|^2 + \frac{1}{2\varepsilon} 
\| x^i_{\tau_j^i(t)} - x^i_t \|^2\Big]
\\
\le {} &
\frac{1}{n} \sum_{i,j=1}^n
\Big[\frac{\varepsilon}{2} 12M_0^2 d_i
+ 
\frac{1}{2\varepsilon} 
\eta ^2 b_{ij}^2 d_i\big[ 24M_0^2 + 27 M_g^2 C^2 \big]
\Big]
\\
= {} &
\frac{\varepsilon}{2}12M_0^2 d + 
\frac{1}{2\varepsilon} \eta^2 \bar {\mathfrak b}^2 d(24M_0^2 + 27M_g^2 C^2)
.
\end{aligned}
\end{equation*}
Taking $\varepsilon = 
\frac{\eta \bar {\mathfrak b} \sqrt d (24M_0^2 + 27M_g^2 C^2)^{\frac{1}{2}}}
{2\sqrt 3 M_0 \sqrt d}
$, we have
\begin{equation}
\begin{aligned}
\label{proof43}
\mathbb E[\frac{1}{n}\sum_{i,j=1}^n \langle D_j(\tau^i_j(t))z^i_{\tau^i_j(t)}, x^i_{\tau^i_j(t)}-x^i_t \rangle] 
\le 
2\sqrt{3d} \eta \bar {\mathfrak b} d  M_0  \big[24M_0^2 +27 M_g^2  C^2 \big] ^\frac{1}{2}.
\end{aligned}
\end{equation}
Now, since $H_i$ is an unbiased estimator of $\nabla g_i^u$, we have
\begin{equation*}
\begin{aligned}
\label{proof45}
\mathbb E[ \langle  (H_i(t) - \nabla g^u_i(x_t^i))^Ty_{t+1}^i , x^i - x_t^i \rangle ] &=
0
.
\end{aligned}
\end{equation*}
By summarizing the previous results, we obtain \eqref{proof36}.

For $\|\delta_t^{G_i}\|^2$, we have
\begin{equation*}
\begin{aligned}
\mathbb E[\|\delta_t^{G_i} \|^2]
&= \mathbb E[ \| (G_0^i(t) - \frac{\partial f^u_0}{\partial x^i}(x_t)) + 
(H_i(t)-\nabla g^u_i(x_t^i))^T{y_{t+1}^i} \|^2   ]
\\
&\le 2\mathbb E [ \| G_0^i(t) - \frac{\partial f^u_0}{\partial x^i}(x_t)  \|^2    ] +
2\mathbb E [ \|(H_i(t)-\nabla g^u_i(x_t^i))^T{y_{t+1}^i} \|^2  ]
\\
&
\le 4\mathbb E\Big[ \|G_0^i(t) \|^2 + \| \frac{\partial f^u_0}{\partial x^i}(x_t) \|^2  \Big]+\sum_{j=1}^m
2C^2 [u^2L_{ij}(d_i+6)^3 + 4(d_i + 5)M_{ij}^2]
\\
&\le 
48 M_0^2 d_i + 4 M_0^2 + 2C^2 [ u^2L_i^2(d_i+6)^3 + 4(d_i+5)M_i^2]
.
\end{aligned}
\end{equation*}
We complete our proof by summing over $i$.
\end{proof}

\begin{lemma}
Let $u\le \frac{M_g}{(d+6) L_g}$ and set $\theta_t = 1, \mu_t = \mu$,  for all $t \ge 0$, we have
\begin{equation}
\mathbb E\Big[ \sum_{i=1}^n 
\langle q_t^i, p_t^i - y_t^i \rangle
\Big]
\le 
\frac{\mu}{1-\rho}(6nZ^2 + 3M_g^2 \bar R + 243 \bar R d M_g^2)
\label{ieq_consensus}
\end{equation}
\end{lemma}
\begin{proof}
From the first-order optimality condition of (\ref{eq:modified_dual}), we have
\begin{equation*}
\left\langle -s_t^i + \frac{1}{\mu_t} (y_{t+1}^i- p_t^i), y - y^i_{t+1} \right\rangle \ge 0
,
\qquad\forall y:y\geq 0,\|y\|\leq C.
\end{equation*}
By plugging in $y = p^i_t$, we have
\begin{equation}
\langle -s_t^i, p_t^i - y_{t+1}^i \rangle \ge 
\frac{1}{\mu_t} \| p_t^i - y_{t+1}^i \|^2,
\end{equation}
and applying the Cauchy-Schwarz inequality leads to
\begin{equation}
\begin{aligned}
\| p_t^i - y_{t+1}^i \| \le \mu_t\|s_t^i \|
.
\end{aligned}
\end{equation}
Next, by the definition of $s_i^t$, we have
\begin{equation*}
\begin{aligned}
\mathbb E[\|s_t^i\|^2]  
\le {} & 
\mathbb E[ \|g_i(x_{t-1}^i) + G_i(t-1)(x_t^i - x_{t-1}^i) + \theta_t [g_i(x_{t-1}^i) - g_i(x_{t-2}^i) ] 
\\
& +
\theta_t[G_i(t-1)(x_t^i - x_{t-1}^i) - G_i(t-2) (x_{t-1}^i - x_{t-2}^i) ]  \|^2]
\\
\le {} & 
3\mathbb E[ \|g_i(x_{t-1}^i) + G_i(t-1)(x_t^i - x_{t-1}^i) \|^2] 
\\
& +
3\theta_t^2\mathbb E[ \| g_i(x_{t-1}^i) - g_i(x_{t-2}^i) \| ^2]
\\
& +
3\theta_t^2\mathbb E[ \|G_i(t-1) (x_t^i - x_{t-1}^i) - G_i(t-2)(x_{t-1}^i  - x_{t-2}^2)\|^2 ]
.
\end{aligned}
\end{equation*}
Recall that $\|g_i(x^i)\| \le Z$, and it follows that
\begin{equation*}
\begin{aligned}
3\mathbb E[ \|g_i(x_{t-1}^i) + G_i(t-1)(x_t^i - x_{t-1}^i) \|^2] 
&\le 6Z^2 + 6\bar R_i^2(\frac{u^2}{2}L_{i}^2(d_i+6)^3 + 2(d_i+4)M_{i}^2)
\\
3\theta_t^2\mathbb E[ \| g_i(x_{t-1}^i) - g_i(x_{t-2}^i) \| ^2]&
\le 3\theta_t^2 M_i^2 \bar R_i^2
\\
3\theta_t^2\mathbb E[ \|G_i(t\!-\!1) (x_t^i \!-\! x_{t-1}^i) - G_i(t\!-\!2)(x_{t-1}^i\!-\!x_{t-2}^2)\|^2]
&\le 12\theta_t^2\bar R_i^2[ \frac{u^2}{2}L_i^2(d_i+6)^3 + 2(d_i+4)M_i^2]
.
\end{aligned}
\end{equation*}
Define $e^i_t = y_{t+1}^i - p^i_t$. Since $\theta_t = 1$ and $u\le \frac{M_g}{(d+6)L_g}$, for all $t \ge 0$ we have
\begin{equation*}
\begin{aligned}
\mathbb E[\|e_t^i \|^2] & \le 
\mu^2 (6Z^2 + 3M_i^2\bar R_i^2 +
243 \bar R_i^2 d_i M_g^2 ) = (\alpha^i)^2
.
\end{aligned}
\end{equation*}
Summing over $i$, we obtain
\begin{equation*}
\begin{aligned}
\sum_{i=1}^n \mathbb E[\|e_t^i \|^2] & \le 
\mu^2 (6nZ^2 + 3M_g^2\bar R + 243 \bar R d M_g^2) = \alpha^2.
\end{aligned}
\end{equation*}
Define the matrices
\begin{equation*}
\begin{aligned}
Y_t = \begin{pmatrix}
- {y_t^1} ^T -
\\
\vdots
\\
-{y_t^n}^T-
\end{pmatrix}
, 
P_t = \begin{pmatrix} 
-{p_t^1 }^T -
\\
\vdots
\\
-{p_t^n}^T-
\end{pmatrix}.
\end{aligned}
\end{equation*}
It is easy to see that
\begin{equation*}
\begin{aligned}
Y_{t+1} & = WY_t  + (Y_{t+1} - WY_t)
\\
\bar Y_{t+1} & = \frac{1}{n} \mathbf 1\mathbf 1^TY_{t+1}
\\
W(W-I) 
& =
(W-\frac 1n \mathbf 1\mathbf 1^T)(W-I) .
\end{aligned}
\end{equation*}
Thus
\begin{equation*}
\begin{aligned}
P_{t+1} -Y_{t+1} &= (W-I) Y_{t+1} 
\\
&= W(W-I)Y_t + (W-I)(Y_{t+1} - WY_t)
\\
&= (W - \frac1n \mathbf 1\mathbf 1^T)(W-I)Y_t + (W-I)(Y_{t+1} - WY_t)
.
\end{aligned}
\end{equation*}
Consequently,
\begin{equation*}
\|P_{t+1}-Y_{t+1}\|_F \le \left\| W- \frac{1}{n}\mathbf 1\mathbf 1^T \right\|_2 \|(W-I)Y_t \|_F + \| W-I\|_2 \| Y_{t+1} - WY_t \|_F.
\end{equation*}
Since $W$ is also a positive semi-definite matrix, we can check that $\|W-I \|_2 \le 1 $.
Define
\begin{equation*}
\begin{aligned}
& \rho  = \left\|W - \frac{1}{n}\mathbf 1 \mathbf 1^T \right\|_2 < 1
\\
& \kappa_t = \| P_{t} - Y_t \|_F
.
\end{aligned}
\end{equation*}
Then, we can derive that
\begin{equation*}
\begin{aligned}
\kappa_{t+1} &\le \rho \kappa_t  + \varphi \alpha
\\
& \le \rho[\rho\kappa_{t-1} + \varphi \alpha ] + \varphi\alpha
\\
& \le \cdots
\\
&\le \rho^{t+1} \kappa _0 + \varphi \sum_{k=0}^{t} \rho^{k} \alpha
.
\end{aligned}
\end{equation*}
Since $y_0^i = 0$ for all $i = 1, ..., n$, we have
\begin{equation*}
\begin{aligned}
&\kappa _{t+1} \le \frac{\varphi\alpha}{1-\rho} \le \frac{\alpha }{ 1-\rho}
\\
&\kappa _{t+1}^2 \le  \Big( \frac{\alpha }{1-\rho} \Big)^2.
\end{aligned}
\end{equation*}
Thus,
\begin{equation*}
\begin{aligned}
\|P_{t+1} - Y_{t+1} \|_F^2 = \sum_{i=1}^n \|p^i_{t+1} - y^i_{t+1}\|^2 \le \Big(\frac{\alpha }{1-\rho}\Big)^2
.
\end{aligned}
\end{equation*}
Then,
\begin{equation*}
\begin{aligned}
\mathbb E[\|q_t^i \|^2] &= \mathbb E[\| \ell _G^i(t) - \ell _G^i(t-1) \|^2]
\\
&= \mathbb E[\| g_i(x_{t-1}^i) - g_i(x_{t-2}^i) + G_i(t-1)(x_t^i-x_{t-1}^i) - G_i(t-2)(x_{t-1}^i - x_{t-2}^i)  \|^2 ]
\\
&\le 3 \mathbb E[\|g_i(x_{t-1}^i) - g_i(x_{t-2}^i))\|^2] + 3\bar R_i^2 \mathbb E [\|G_i(t-1)\|^2]  + 3\bar R_i^2 \mathbb E[\|G_i(t-2)\|^2]
\\
&\le  3\bar R_i^2 M_i^2 + 81\bar R_i^2  d_i M_g^2
.
\end{aligned}
\end{equation*}
Using the Peter--Paul inequality, we get
\begin{equation}
\begin{aligned}
\mathbb E\Big[\sum_{i=1}^n \langle q_t^i, p_t^i - y_t^i \rangle \Big] 
\le {} & 
\mathbb E\bigg[ \sum_{i=1}^n \Big( \frac{\|q_t^i\|^2}{2\varepsilon} + \frac{\varepsilon}{2}\|p_t^i - y_t^i \|^2  \Big)  \bigg] 
\\
= {} &\frac{1}{2\varepsilon}
\sum_{i=1}^n\Big(3\bar R_i^2 M_i^2 + 
81\bar R_i^2  d_i M_g^2 
\Big)+\frac{\varepsilon}{2}(\frac{\alpha}{1-\rho})^2 
\\
= {} &
\frac{1}{2\varepsilon}
\Big(3\bar R^2 M_g^2 + 
81\bar R^2  d M_g^2 
\Big)
\\
+ {} &
\frac{\varepsilon}{2}
(\frac{1}{1-\rho})^2 [\mu^2 (6n Z^2 + 3M_g^2\bar R + 243 \bar R d M_g^2) ].
\end{aligned}
\end{equation}
Setting $\varepsilon = 
\frac{
(1-\rho)(3\bar R^2 M_g^2 + 81 \bar R d M_g^2 )
^{\frac{1}{2}}
}
{
\mu(6nZ^2 + 3M_g^2 \bar R + 243 \bar R d M_g^2)
^{\frac{1}{2}}
}
$, we now conclude that
\begin{equation}
\begin{aligned}
\mathbb E\Big[\sum_{i=1}^n \langle q_t^i, p_t^i - y_t^i \rangle \Big]
&\le 
\frac{\mu}{1-\rho} (3\bar R^2 M_g^2 + 81 \bar R d M_g^2 )^{\frac{1}{2}}
(6nZ^2 + 3M_g^2 \bar R + 243 \bar R d M_g^2)
^{\frac{1}{2}}
\\
&\le 
\frac{\mu}{1-\rho}(6nZ^2 + 3M_g^2 \bar R + 243 \bar R d M_g^2).
\end{aligned}
\end{equation}
This completes the proof.
\end{proof}

With all the necessary preparations in place, we now proceed to prove our results.
\begin{proof}[Proof of Theorem \ref{theorem:1}]
% u \le M_g / (d+6)L_g
Given that $G_{ij}(t) = \frac{g_{ij}(x^i_t + u\hat z^i_t) - g_{ij}(x^i_t-u\hat z)^i_t}{2u} \hat z^i_t$, 
based on \cite{nesterov2017random}, we have
\begin{equation*}
\begin{aligned}
\mathbb E[\| G_{ij}(t-1) - \nabla g^u_{ij}(x_{t-1}^i) \|^2 ]
&
\le 2\mathbb E[\|G_{ij}(t)\|^2]   + 2 \|\nabla g^u_{ij}(x_{t-1}^i)\|^2
\\
&
\le  2  \Big(\frac{u^2}{2} L_{ij}^2 (d_i+6)^3 + 2(d_i+4) \| \nabla g_{ij}(x)  \|^2       \Big)      + 2M_{ij}^2
\\
& 
\le u^2L^2_{ij} (d_i+6)^3 + 4(d_i+4) M_{ij}^2 + 2M_{ij}^2
\\
&\le u^2L^2_{ij}(d_i+6)^3 + 4(d_i + 5)M_{ij}^2
.
\end{aligned}
\end{equation*}
This leads to the following inequality
\begin{equation}
\begin{aligned}
\mathbb E[\| G_i(t-1) - \nabla g_i^u (x_{t-1}^i) \|_F ^2 ] 
&\le 
\sum_{j=1}^m
\Big[ u^2 L^2_{ij} (d_i+6)^3 + 4(d_i+5) M_{ij}^2  \Big]
\\
&=u^2L_i^2(d_i+6)^3 + 4(d_i+5)M_i^2
\le 
31 d_i M_g^2 
.
\end{aligned}
\end{equation}
According to Lemma \ref{lemma:3},we obtain
\begin{equation*}
\begin{aligned}
\mathbb E[ \| g_i(x_{t-1}^i) - g_i^u(x_{t-1}^i)  \|^2 ] &
=
\mathbb E[ \sum_{j=1}^m (g_{ij}(x_{t-1}^i) - g_{ij}^u(x_{t-1}^i))^2 ]
\\
&\le  
\frac{1}{4}\sum_{j=1}^m  
u^4L_{ij}^2d_i^2
=\frac{1}{4}u^4 L_i^2 d_i^2
\end{aligned}
\end{equation*}
where the last equality holds because
$\sum_{j=1}^m L_{ij}^2 = L_i^2$.
Then, we have
\begin{equation}
\begin{aligned}
\mathbb E [ \|\delta_{t}^{F_i}\|_2^2 ] &= \mathbb E[\|\ell_{G}^i(t) -
\ell_{g^u}^i(t)\|_2^2]
\\
&=\mathbb E\Big[\|(g_i(x^i_{t-1}) + G_i(t-1)(x^i_t - x^i_{t-1})) - ( g^u_i(x^i_{t-1}) + \nabla g^u_i(x^i_{t-1})(x^i_t - x^i_{t-1}))\|_2^2 \Big]
\\
&= \mathbb E \Big[\| ( g_i(x_{t-1}^i) - g_i^u(x_{t-1}^i) ) + (G_i(t-1) -\nabla g_i^u(x_{t-1}^i)(x_t^i - x_{t-1}^i) 
)\|_2^2
\Big]
\\
&\le 2 \mathbb E[\|g_i(x_{t-1}^i)-g_i^u (x_{t-1}^i)\|^2 ] + 2\mathbb E[\| G_i(t-1) - \nabla g_i^u (x_{t-1}^i) (x_t^i - x_{t-1}^i)   \|^2 ]
\\
&\le u^4 L_i^2 d_i^2
+ 62 d_i M_g^2 \bar R_i^2
.
\end{aligned}
\end{equation}
Summing over $i$, we have
\begin{equation}
\begin{aligned}
\label{ieq_delta_F_2}
\sum_{i=1}^n \mathbb E[\|\delta_t^{F_i} \|_2^2]\le 
u^4 L_g^2 d^2 +62dM_g^2 \bar R^2.
\end{aligned}
\end{equation}
Then, we note that
\begin{align*}
\mathbb E[\langle \delta_{t+1}^{F_i}, y_{t+1}^i-y \rangle] &= 
\mathbb E\big[ \langle \mathbb E[\delta_{t+1}^F | \mathcal F_{t}], y_{t+1}^i-y\rangle\big],
\end{align*}
where
\begin{equation*}
\begin{aligned}
\mathbb E[ \delta_{t+1}^{F_i} | \mathcal F_{t}] &= \mathbb E[g_i(x_t^i) - g_i^u(x_t^i) | \mathcal F_t ] 
\le \frac{u^2d_i}{2} L_i 
.
\end{aligned}
\end{equation*}
Thus
\begin{equation*}
\begin{aligned}
\mathbb E[\langle \delta_{t+1}^{F_i}, y_{t+1}^i-y \rangle] &= 
\mathbb E\big[ \langle \mathbb E[\delta_{t+1}^F | \mathcal F_{t}], y_{t+1}^i-y\rangle\big]
\le {u^2d_i} L_i  C
.
\end{aligned}
\end{equation*}
Summing over $i$ leads to
\begin{align}
\sum_{i=1}^n \mathbb E[\langle \delta_{t+1}^{F_i}, y_{t+1}^i-y \rangle] \le {u^2 d L_{\max} C}
,
\label{ieq_delta_F_deltay}
\end{align}
where $L_{\max} = \max\{L_1, L_2, ..., L_n \}$.

Next, we note that
\begin{equation}
\begin{aligned}
\sum_{i=1}^n 
\mathbb E[ \|q^i_t - \bar q^i_t\|_2^2 ] 
&= 
\sum_{i=1}^n \mathbb E [  \|
\ell_{G}^i(t)  - \ell_{G}^i({t-1}) 
- \ell_{g^u}^i(t) + \ell_{g^u}^i({t-1}) \|^2_2   ]
\\
&\le \sum_{i=1}^n( 2\mathbb E[\|\delta_{t}^{F_i} \|_2^2] + 2\mathbb E[ \| \delta_{t-1}^{F_i} \|_2^2 ])
\\
& \le  (\sum_{i=1}^n 4u^4L_i^2d_i + 248d_i M_g^2 \bar R_i^2)
\\
&\le 4u^4 L_g^2 d^2 + 248 d M_g^2 \bar R^2
.
\label{ieq_q_t_2}
\end{aligned}
\end{equation}
Setting $\theta_t = 1$, $\mu_t = \mu$, $\gamma_t = 1$ and 
$\frac{1}{\eta_t} = L_0 + L_{\max} + \frac{1}{\eta}$, and
combining the results from 
\eqref{initial_conver_analy},
\eqref{proof36},
\eqref{proof37}, 
\eqref{ieq_delta_F_deltay},
\eqref{ieq_q_t_2} and
\eqref{ieq_consensus},
we arrive at the following inequality
\begin{equation*}
\begin{aligned}
&\mathbb E[f_0(\bar x_T) - f_0(x^*) ]
\\
\le {} & 
\frac{1}{T}(L_0 + L_{\max})\bar R^2 + \frac{1}{T\eta} \bar R^2 + \eta(104M_0^2d + 124M_g^2dC^2) + \mu(6u^4L_g^2d^2 + 372 d M_g^2 
\bar R^2)
\\&
+\eta(M_0 \bar {\mathfrak b} \sqrt d +
2\sqrt 3 \bar {\mathfrak b} d M_0)(24M_0^2 + 27M_g^2C^2) ^ {\frac{1}{2}}
+u^2 d L_{\max} C 
\\
&+\frac{1}{2} u^2 L_0 d + \frac{\mu}{1-\rho}(6dZ^2 + 3M_g^2 \bar R + 243 \bar R d M_g^2)
.
\end{aligned}
\end{equation*}
Similarly, by combining \eqref{initial_viola_analy},
\eqref{proof36},
\eqref{proof37},
\eqref{ieq_consensus},
\eqref{ieq_delta_F_2},
\eqref{ieq_delta_F_deltay},
\eqref{ieq_q_t_2},
we derive the following result
\begin{equation*}
\begin{aligned}
&\mathbb E\!\left[\left\| \left[ \sum_{i=1}^n g_i(\bar x_T^i)  \right]_+   \right\| \right]
\\
\le {} & 
\frac{1}{T}(L_0 + L_{\max} )\bar R^2 
+ \frac{1}{T\eta} \bar R^2 + \frac{1}{\mu} n C^2
+\eta (104M_0^2 d + 124M_g^2 d C^2) 
\\
&
+ \mu(7u^4L_g^2 d^2 + 403 d M_g^2 \bar R^2)
+\eta(M_0 \bar {\mathfrak b} \sqrt d +
2\sqrt 3 \bar {\mathfrak b} d M_0)(24M_0^2 + 27M_g^2C^2) ^ {\frac{1}{2}}
\\
&
+u^2dL_{\max} C + \frac{\mu}{1-\rho}(6dZ^2 + 3M_g^2 \bar R + 243 \bar R d M_g^2)
+\eta L_g\bar R C
.
\end{aligned}
\end{equation*}

To refine these results further, we set the parameters as $\eta = \frac{\bar R}{\sqrt {T \xi}}$ and $\mu = \frac{C\sqrt {2n}}{\sqrt {T\zeta}}$, where the constants $\xi$ and $\zeta$ are given by
\[
\begin{aligned}
\xi ={} &
(M_0\bar {\mathfrak b} \sqrt d + L_0\bar b d \bar R + 2\sqrt 3 \bar {\mathfrak b} d M_0)(24M_0^2 + 27M_g^2 C^2)^{\frac{1}{2}}
+ 104M_0^2d + 124M_g^2 dC^2, \\
\zeta ={} &
 403d M_g^2\bar R + \frac{1}{1-\rho}(
6dZ^2 + 3M_g^2 \bar R + 243 \bar RdM_g^2
).
\end{aligned}
\]
With these parameter choices, we subsequently derive the convergence rate and constraint violation bounds as specified in \eqref{result_convergence_rate} and \eqref{result_violation}.
\end{proof}